\documentclass{amsart}
\usepackage{%amsfonts,	amsmath,
amssymb,enumitem,color
}

\newcommand{\li}{\operatorname{Li}}
\newcommand{\Gal}{\operatorname{Gal}}

\newcommand{\norm}{\unlhd}
\newcommand{\Occ}{\operatorname{Occ}}
\newcommand{\GL}{\operatorname{GL}}
\newcommand{\SL}{\operatorname{SL}}
\newcommand{\PSL}{\operatorname{PSL}}

\def\cA{\mathcal A}      \def\cF{\mathcal F}        \def\cN{\mathcal N}   \def\cP{\mathcal P}       
    \def\sl{\mathscr l}  
\def\fd{\mathfrak{d}}
\def\ff{\mathfrak{f}}

\def\fO{\mathfrak{O}}
\def\Z{{\mathbb Z}}  \def\F{{\mathbb F}}   \def\Q{{\mathbb Q}}

\newtheorem{lemma}{Lemma} 
\newtheorem{prop}{Proposition}
\newtheorem{corollary}{Corollary} 
\newtheorem{theorem}{Theorem} 
\newtheorem{conj}{Conjecture} 
 
\newtheorem{rem}{Remark} 
\newtheorem{ex}{Example}

\def\su#1{\sum_{\substack{#1}}}
\def\pr#1{\prod_{\substack{#1}}}

\def\bs#1{\begin{equation*} \begin{split} #1 \end{split} \end{equation*}}
\def\bsc#1{\begin{equation} \begin{split} #1 \end{split} \end{equation}}
\def\eqs#1{\begin{equation*} #1 \end{equation*}}
\def\eqn#1{\begin{equation} #1 \end{equation}}
\def\mult#1{\begin{multline*}#1\end{multline*}}
\def\multn#1{\begin{multline}#1\end{multline}}

\def\({\left(} \def\){\right)} \def\[{\left[} \def\]{\right]} 
\def\fl#1{\left\lfloor#1\right\rfloor} \def\ceil#1{\left\lceil#1\right\rceil}
\def\le{\leqslant} \def\ge{\geqslant}
\def\eps{{\varepsilon}}

\definecolor{orange}{rgb}{0.7,0.3,0}
\definecolor{blue}{rgb}{.2,.6,.75}

\def\mod#1{\,\text{mod }#1}
\def\Im{\text{Im}}

\def \pE {\pi_E(x; f, a)}
\def\E {\widetilde{E}}
\def \cq#1{\Q (\zeta_{#1})} % Cyclotomic Extension
\def \ab #1{K_{#1}\textsuperscript{ab}} %maximal abelian subextension
\def \gaq#1{\Gal(#1/\Q)} %Galois groups over \Q
\def \me{M_E}
\def \gl#1{\GL({#1})} 
\def \sl#1{\SL({#1})}
\def \psl#1{\PSL_2({#1})}

\begin{document}
\title[Cyclicity Conjecture]{Cyclicity of Elliptic Curves Modulo Primes in arithmetic progressions}

\author{Y\i ld\i r\i m Akbal}
\address{Department of Mathematics \\At\i l\i m University \\06830 G\"{o}lba\c{s}\i, Ankara, TURKEY
}
\email{yildirim.akbal@atilim.edu.tr} 

\author{Ahmet M.~G\"ulo\u glu}
\address{Department of Mathematics \\Bilkent University \\06800 Bilkent, Ankara, TURKEY}
\email{guloglua@fen.bilkent.edu.tr}

%	\thanks{Both authors were supported by T\"UB\.ITAK Research Grant no. 114F404}

\subjclass[2010]{Primary 11G05; Secondary 11N13, 11N36, 11N45, 11R45}
\date{\today}
\keywords{Serre's Cyclicity Conjcture, Reduction of Elliptic Curves modulo primes, Primes in Arithmetic Progressions, Chebotarev Density Theorem}

\begin{abstract} 
We consider the reduction of an elliptic curve defined over the rational numbers modulo primes in a given arithmetic progression and investigate how often the subgroup of rational points of this reduced curve is cyclic as a special case of Serre's Cyclicity Conjecture. 
\end{abstract}

\maketitle

\section{Introduction}
\subsection{History of the Cyclicity Conjecture}
Let $E/\Q$ be an elliptic curve given by a global minimal (see \cite[Corollary~VIII.8.3]{Sil86}) Weierstrass equation 
\eqs{y^2 + a_1xy + a_3y = x^3 + a_2x^2 + a_4x + a_6,}
where $a_1, \ldots, a_6 \in \Z$. Primes that do not divide the discriminant $\Delta_E$ of this equation, or equivalently, its conductor $N_E$, are called the primes of good reduction. For such primes $p$, the reduction $\E_p$ of $E$ modulo $p$ is a non-singular elliptic curve. In particular, let $\E(\F_p)$ denote the subgroup of $\F_p$-rational points of the reduced curve $\E_p$.

In 1976, S.~Lang and H.~Trotter formulated (cf.~\cite{LT}) the following elliptic curve analogue of Artin’s primitive root conjecture:
\begin{conj}[Lang-Trotter Conjecture]
Let $E/\Q$ be an elliptic curve of rank at least 1. Let $P \in E(\Q)$ be a fixed point on $E$ of infinite order. Then, the density of primes such that $\E(\F_p) = \langle P \mod p \rangle$ exists.
\end{conj}
As the first step towards this conjecture, the same year, following Hooley's conditional proof of Artin's conjecture (cf. \cite[Ch.~3]{Hooley}), Jean Pierre Serre proved (cf.~\cite{Serre}) assuming GRH that
\eqn{\label{SerreAsym}
\big|\{ p \le x : p \nmid N_E, \E(\F_p) \text{ is cyclic}\}\big| = \delta_E \li(x) + o(x/\log x),}
with the density $\delta_E$ given by
\eqn{\label{DeltaE}
\delta_E  = \sum_{n \ge 1 }\frac{\mu(n)}{[K_n:\Q]}.}
Here, $\li(x)=\int_2^x 1/\log t dt$, and $K_n=\Q(E[n])$ is the $n$-\emph{division field} obtained by adjoining to $\Q$ the affine coordinates of the group $E[n](\overline{\Q})$ of $n$-torsion points of $E$, where $\overline{\Q}$ is a fixed algebraic closure of $\Q$.

Serre, in \cite{Serre}, does not show, however, that $\delta_E >0$, but leaves it as an exercise! Murty and Cojocaru have shown in \cite[pp.~621-2]{CoMu} that $\delta_E>0$ for both CM and non-CM curves, provided $K_2 \neq \Q$. This result also follows as a byproduct of Theorem \ref{T:DensitynonCM} below by taking $f=1$ for non-CM curves, and provides an important modification needed in their argument for the non-CM case (see Remark \ref{CoMuCorrection}). All of these results depend on GRH.

In general, an explicit Euler product for $\delta_E$ is known only for the so-called Serre curves (see, for example, \cite[\S~2.4.1]{Brau}, both for the definition and the explicit formula for $\delta_E$). 

Serre, possibly motivated by Lang-Trotter conjecture, also claimed in \cite{Serre}: 
\begin{conj}[Serre's Cyclicity Conjecture]
$\E(\F_p)$ is cyclic for infinitely many primes $p$ if and only if $E$ contains a non-rational 2-torsion point.
\end{conj}
In 1990, Gupta and R. Murty showed in \cite{GM} that for any elliptic curve $E$, $\E(\F_p)$ is cyclic for at least $c_E x/(\log x)^2$ primes for some positive constant $c_E$, provided $K_2 \neq \Q$. When $K_2 = \Q$, then the torsion group $E(\Q)_{\text{tors}}$ of rational points on $E$ contains a subgroup of the form $\Z/2\Z \times \Z/2\Z$. Since for all primes $p$, except for a finite number of them, the torsion group embeds into $\E(\F_p)$, we deduce that there are only finitely many primes $p$ for which $\E(\F_p)$ is cyclic, thereby settling Serre's cyclicity conjecture.

The asymptotic formula \eqref{SerreAsym}, however, has been proven \emph{unconditionally} only for curves with complex multiplication (that we shall call CM curves). In 1979, Ram Murty showed (cf.~\cite{Mu1979}) that \eqref{SerreAsym} holds without GRH for all CM elliptic curves. In 2010, Akbary and K. Murty improved (cf.~\cite[Thm 1.1]{Akbary}) the error term  of \cite{Mu1979} to $O(x/(\log x)^A)$ for any sufficiently large positive constant $A$.  They, however, assume that the curve has multiplication by the full ring of integers $\fO_K$ of an imaginary quadratic field $K$. For non-CM curves, A.~C.~Cojocaru showed (cf.~\cite{Co1}) in 2002 that if $E$ is a non-CM elliptic curve, then \eqref{SerreAsym} holds with an error $\ll_{N_E} x\log \log x/(\log^2 x)$ under the assumption that the Dedekind zeta functions of the division fields of $E$ have no zeros to the right of $x=3/4$. 

Upon combining the results of \cite{Mu1979,Akbary,GM}, it follows that $\delta_E >0$ for curves with complex multiplication by $\fO_K$, which gives a second proof of Serre's conjecture for these curves via the asymptotic formula \eqref{SerreAsym}.

In 2004, assuming GRH, Cojocaru and Murty improved (cf.~\cite{CoMu}) the error terms in \eqref{SerreAsym} to $O_{N_E}(x^{5/6}(\log x)^{2/3})$ for non-CM curves, and to $O(x^{3/4}(\log N_Ex)^{1/2})$ for CM curves with explicit dependence on the conductor $N_E$. This way, they were able to deduce estimates for the smallest prime $p_E$ for which $\E(\F_p)$ is cyclic.
\subsection{The goal of this paper}
For the rest of the paper, $f\ge 1$ is an integer, and $a$ represents a residue class modulo $f$ and $(a,f)=1$. 

We consider Serre's cyclicity conjecture for primes $p \equiv a \mod f$. More precisely, for a given elliptic curve $E$, we try to determine all moduli $f$, and the corresponding residue classes $a$ for each modulus $f$ such that $\E(\F_p)$ is cyclic for infinitely many primes $p \equiv a \mod f$. 

Answering this question in the most general setting with any modulus $f$, any residue class $a$ and an arbitrary elliptic curve $E$ turns out to be too ambitious. Unfortunately, we cannot provide a complete answer to what we seek. The main difficulty is that the non-trivial intersections of the division fields $K_n$ for an arbitrary elliptic curve are not completely understood. This is exactly the same reason why there is no explicit product in general for $\delta_E$ in \eqref{SerreAsym}. On the other hand, we do have conditional and unconditional results, which partially complement each other, and a conjecture which we believe gives the correct answer. We find asymptotic formulas under GRH with error terms similar to the ones given by Cojocaru and Murty in \cite{CoMu} mentioned above, and with explicit dependence on the modulus $f$ and certain constants related to the curve $E$, but the main obstacle in this case is to show that the corresponding density, which we shall denote $\delta_E (f,a)$, is positive. We also give unconditional lower bound estimates similar to the one given by Gupta and Murty in \cite{GM}.

Before we state our prediction, we first introduce some notation. We denote by $\zeta_n$ any fixed primitive $n$\textsuperscript{th} root of unity, and by $\cq n$ the corresponding cyclotomic extension. The letter $\sigma$ when used with a subscript is reserved for automorphisms of cyclotomic fields and the one which takes $\zeta_n$ to $\zeta_n^a$, for each $a$ coprime to the modulus in question, will be denoted by $\sigma_a$. Also, the letters $p$ and $q$ always denote primes.  

\begin{conj}\label{ourConj}
Let $E$ be an elliptic curve over $\Q$ and let $f$ and $a$ be relatively prime positive integers. Then, there are infinitely many primes $p \equiv a \mod f$ for which $\E(\F_p)$ is cyclic unless $K_d \subseteq \cq f$ for some $d\ge 2$ and $\sigma_a \in \Gal(\cq f /K_d)$, in which case there are at most a finite number of such primes.
\end{conj}

One direction follows easily. To see this, we first need to quote two key facts from \cite[Lemma 2.1, Prop. 3.5.3]{CoMu}:
\begin{enumerate}
\item[{\bf 1.}] For \emph{odd} $p\nmid N_E$,  $\E(\F_p)$ is cyclic if and only if $p$ does not split completely in $K_q$ for any prime $q \neq p$. 

\item[{\bf 2.}] $\cq n \subseteq K_n$ for each integer $n\ge 2$.  
\end{enumerate} 
Now, if $K_d \subseteq \cq f$ for some $d\ge 2$, and $\sigma_a$ fixes $K_d$, then any $p \nmid N_E$ with $p \equiv a \mod f$ will split completely in $K_d$, thereby in any $K_q$ with $q\mid d$. Thus, $\E(\F_p)$ cannot be cyclic for odd $p\neq q$ with $p \nmid N_E$. We record this result below. But, first note that $K_d \subseteq \cq f$ implies $K_d$ is abelian over $\Q$, and  Gonz\'alez–Jim\'enez and Lozano-Robledo show (cf.~\cite{GonLoz}) that $K_d$ is abelian only if $d \in \{2,3,4,5,6,8\}$ for non-CM curves, and if $d \in \{2,3,4\}$ for CM curves. Thus, we deduce the following result.
\begin{prop}\label{ConjEZpart}
Assume that $(a,f)=1$, $K_d \subseteq \cq f$ for some $d \in \{2,3,4,5,6,8\}$, and $\sigma_a$ fixes $K_d$. Then, $\E(\F_p)$ is cyclic for at most finitely many primes $p \equiv a \mod f$. 
\end{prop}
Note that when $K_d$ is abelian for some $d \in \{2,3,4,5,6,8\}$, then $K_d \subseteq \cq f$ exactly when the conductor $\ff_d$ of $K_d$ divides $f$ (see the beginning of section \ref{uncondresults}). In general, it may not be easy to determine $\ff_d$. On the other hand, it is easy to determine all moduli $f$ for which $K_d \subseteq \cq f$ if $K_d = \cq d$. Gonz\'alez–Jim\'enez and Lozano-Robledo give complete classification and parametrization of all elliptic curves $E/\Q$, up to isomorphism over $\Q$, such that $K_d$ is abelian over $\Q$, and those curves such that $K_d = \cq d$. Furthermore, they classify all the abelian Galois groups $\gaq{K_d}$ for each $d \ge 2$ that may occur. In particular, they show that $K_d = \cq d$ only when $d\in\{2,3,4,5\}$ for non-CM curves, and $d\in \{2,3\}$ for CM curves. Thus, it follows from their result that $\cq d = K_d \subseteq \cq f$ only if $d \mid f$, unless $d=2$ (which is the trivial case since $\E(\F_p)$ is not cyclic for all but a finite number of primes by Serre's already proven conjecture 2).

One can say more about $K_2$ when the Weierstrass model 
\eqn{\label{eq:nicemodel}
y^2 = g(x) = x^3 + Ax^2 + Bx + C}
is used for $E$. Indeed, the discriminant of $g(x)$ is given by
\eqs{\Delta = A^2B^2 - 4B^3 -4A^3C - 27C^2 + 18ABC.}
Since the $x$-coordinates of 2-torsion points are the roots of this cubic, it follows from Galois Theory that $K_2$ is non-abelian if and only if $g(x)$ is irreducible and $\Delta$ is not the square of a rational number. If this is the case, then $K_{2d} \not\subset \cq f$ since $K_2 \subseteq K_{2d}$ for any $d\ge 1$. Furthermore, $K_2 = \Q(\alpha,\sqrt{\Delta})$, where $\alpha$ is any root of $g(x)$. If $\Delta \in \Q^2$, then $K_2$ is a cubic abelian extension, and we can find its conductor in this case (see below). If $g$ splits into three linear factors, then $K_2 = \Q$, and if $g$ factors into a linear factor and an irreducible quadratic, then $K_2$ is a quadratic extension of $\Q$, and we can also determine the conductor easily in this case.

In what follows, we list the partial results we can prove that support our prediction in Conjecture \ref{ourConj}. 

\subsection{Unconditional Results} \label{uncondresults}
Let $\ab n$ be the maximal abelian extension of $\Q$ in $K_n$. By the Kronecker-Weber Theorem, $\ab n \subseteq \cq{\ff_n}$ for some positive integer $\ff_n$, minimal with respect to this inclusion, that consists of primes that ramify in $\ab n$. This number $\ff_n$ is called the conductor of $\ab n$. %Note also that these primes also ramify in $K_n$. So, at this point we recall another important fact. \begin{description} \item[Fact 3] Primes that ramify in $K_n$ are the divisors of $nN_E$. \end{description}
\begin{theorem} \label{T:SieveK2cubic}
Let $E$ be an elliptic curve over $\Q$ satisfying $[K_2:\Q]=3$ and let $a$ and $f$ be
any positive integers such that $(a,f)=1$ and $(a-1,f)$ has no odd prime divisors. Let $A \ge 0$ be given. Then, for $x$ sufficiently large and assuming $f \ll (\log x)^A$, the group $\E(\F_p)$ is cyclic for $\gg x/(\log x)^{2+A}$ primes $p \equiv a \mod f$, unless $K_2 \subseteq \cq f$ and $\sigma_a$ fixes $K_2$.
\end{theorem}

To see why this Theorem is consistent with and provides an affirmative answer to Conjecture \ref{ourConj}, note that the Artin map $\langle p, \cq f/\Q \rangle =\sigma_a$ for any prime $p \nmid N_E$ with $p \equiv a \mod f$. Thus, if $K_q \subseteq \cq f$ for some \emph{odd} prime $q$, and $\sigma_a$ fixes $K_q$, then it also fixes $\cq q$, and this means $q \mid (a-1,f)$, contradicting our assumption in Theorem \ref{T:SieveK2cubic}. Therefore, it is enough to check whether $K_q \subseteq \cq f$ and $\sigma_a$ fixes $K_q$ only for $q=2$.

The main advantage of this result compared to an asymptotic formula is that it is unconditional, and works for any elliptic curve, CM or non-CM. More importantly, this gives a positive answer to our conjecture for certain residue classes. It is also practical in the sense that one can determine the moduli $f$, and when $K_2 \not\subset \cq f$, also the residue classes $a$ for which $\E(\F_p)$ is cyclic for infinitely many primes $p \equiv a \mod f$. To see this, note that if $E$ is given by 
\eqs{y^2 = x^3 + a_1x^2 + a_2x + a_3,}
with an irreducible cubic, then $K_2$ is a cubic extension exactly when the discriminant $\Delta_E$ is a square in $\Q$. In this case, H\"aberle describes in \cite[Corollary 12]{Haberle} how to easily determine the conductor $\ff_2$ of a cubic extension of $\Q$. In particular, $\ff_2$ is of the form 
\eqs{q_1 q_2 \cdots q_r \qquad (r \ge 1),}
where each $q_i \equiv 1 \mod 3$ is a prime, with \emph{at most} one exception, which then must be $9$. Therefore, any number $f$ not divisible by $\ff_2$ will be an admissible modulus, and we may then choose the residue class $a$ coprime to $f$ such that $(a-1,f)$ has no odd prime divisors. Furthermore, if $\ff_2 \mid f$, but the order of $a$ modulo $f$ does not divide $\varphi(f)/3 = |\Gal(\cq f/K_2)|$, then $\sigma_a$ cannot fix $K_2$. 

In general, there are $2\varphi(f)/3$ possible choices for $a$. In particular, when $f$ is a prime power divisible by $\ff_2$, one can take any residue class $a$ which is not a cubic residue modulo $f$.

The proof of Theorem \ref{T:SieveK2cubic} uses linear sieve of Iwaniec (cf.~\cite{Iwaniec}). The idea is to count the primes $p\le x$ with $p\equiv a \mod f$ such that $p-1$ is free of odd primes not exceeding $x^\alpha$ for some $\alpha>1/4$. Having the exponent $\alpha>1/4$ is essential for the rest of the proof to work, and one way to achieve this is to combine the linear sieve of Iwaniec with a follow up paper by Iwaniec and Fouvry with a necessary modification provided later by Heath-Brown (see \cite[Lemma 2]{Heath}). Using sieve theory also necessitates the restriction on residue classes in Theorem \ref{T:SieveK2cubic}. Indeed, were some odd prime $q\le x^\alpha$ to divide $(a-1,f)$, $p$ would split completely in $\cq q$; that is, $q \mid p-1$, and one cannot guarantee then that $p$ does not split in $K_q$, which is the only way the sieve can be used to prove Theorem \ref{T:SieveK2cubic}.

Since it is desirable to remove the restriction on residue classes $a$, we also investigated ways to deal with the case when $(a-1,f)$ is divisible by odd primes. To understand the obstacles in this situation, we consider an example. Say, $f>5$ is a prime, and we want to count primes $p \equiv 1 \mod f$ for which $\E(\F_p)$ is cyclic. Note that these primes split completely in $\cq f$. Fortunately, there is hope for these primes not to split completely in $K_f$ since it follows from \cite{GonLoz} that $K_f$ is non-abelian when $f>5$. One has to make sure $p$ does not split completely in $K_q$ for primes $q\neq p$. To get an unconditional result using sieve methods, one has to count primes $p\le x$, $p \nmid N_E$, $p-1$ not divisible by primes $q \le x^\alpha$ with some $\alpha > 1/4$ except for $2$ and $f$, and the Artin map $\langle p,K_{2f}/\Q \rangle \subseteq C$, where $C$ is a conjugacy class that consists of automorphisms in $\Gal(K_{2f}/\cq f) \setminus \{ 1_{K_{2f}}\}$. This may be done using a result of Murty and Petersen (cf. \cite[Theorem 0.2]{MuPe}), but only, in the best scenario, with an exponent $\alpha = 1/2(\varphi(f)-2)-\eps < 1/4$ (note $\varphi(f)=f-1 > 4$). Thus, unless their paper can be improved, getting an unconditional result seems to be out of reach with current methods.

One last note relevant also to the next result is that when applying the sieve one has to work with two congruences; namely, that $p \equiv a \mod f$ and $p \equiv b \mod \ff_2$. The latter is needed to make sure that $p$ does not split completely in $K_2$ (see Lemma \ref{Ram-Gupta} and Remark \ref{R:GM}). When $K_2$ is cubic, these two congruences are shown to be compatible in Lemma \ref{L:K2cubic}, and this leads to Theorem \ref{T:SieveK2cubic} above. However, in what follows, we shall see that this is not always the case when $K_2$ is non-abelian, or a quadratic field. Thus, the next result is slightly weaker than but is similar to the cubic case. 

The character $\chi_D$ that appears in the statement of Theorem \ref{T:SieveK2quad} below is the real primitive character of conductor $|D|$ associated with the quadratic field $\Q(\sqrt D)$ given by the Kronecker symbol $\chi_D(\cdot)= \(\frac{D}{\cdot}\)$, and $\fd_2$ stands for the discriminant of the quadratic extension $\ab 2$ of conductor $\ff_2 = |\fd_2|$. 

In case one uses a Weierstrass model given by \eqref{eq:nicemodel}, $\ab 2$ is generated by the square root of the square-free part of $\Delta_E$. So, in practice, conditions given below can easily be checked to determine which moduli $f$ and the corresponding residue classes $a$ are admissible.

\begin{theorem} \label{T:SieveK2quad}
Let $E$ be an elliptic curve over $\Q$ satisfying $[\ab 2:\Q]=2$ and let $a$ and $f$ be
any positive integers such that $(a,f)=1$ and $(a-1,f)$ has no odd prime divisors. Let $A \ge 0$ be given. Then, for $x$ sufficiently large and assuming $f \ll (\log x)^A$, the group $\E(\F_p)$ is cyclic for $\gg x/(\log x)^{2+A}$ primes $p \equiv a \mod f$ if $\ff_2 \nmid f$, unless $\ff_2=3(f,\ff_2)$ and $\chi_{-\fd_2/3}(a) = -1$. The same lower bound holds if $\ff_2 \mid f$ and $\sigma_a$ does not fix $\ab 2$.
\end{theorem}
Note that Theorem \ref{T:SieveK2quad} comes close to, but falls short of providing the converse of Proposition \ref{ConjEZpart} due to the exceptional case when $\ff_2 \nmid f$. To see what the problem is, we consider an example: 

Assume that $\ab 2 = \Q(\sqrt{21})$, $f=7$, and $a=5$ so that
\eqs{\ff_2 = \fd_2 = 21, \quad \( \frac{-\fd_2/3}{5}\) = \(\frac{-7}{5}\)=-1, \quad 21 = 3\gcd(7,21).}
Since $\ff_2 \nmid f$, $\ab 2 \not\subset \cq f = \cq 7$. We require primes $p \equiv 5 \mod 7$ not split completely in $\ab 2$ so that they do not not split completely in $K_2$. The latter is achieved by imposing a condition that $p \equiv b \mod 21$ for some $b$. We want to see why the sieve cannot be applied. Note that the second congruence should guarantee that $\sigma_b \in \gaq{\cq{21}}$, but $\sigma_b$ does not fix $\ab 2$; that is, $\sigma_b(\sqrt{21}) = - \sqrt{21}$. Here, $b$ should be chosen in such a way that $(b-1,21)=1$. At the same time, we need $7 \mid b-5$ so that the congruences $p \equiv 5 \mod 7$ and $p \equiv b \mod 21$ are compatible. This implies then that $\sigma_b$ restricted to $\cq 7$ sends $\sqrt{-7}$ to $-\sqrt{-7}$ because $\sigma_a=\sigma_5$ does. This can be seen as follows:

\begin{quote}
	The Artin map $\langle 5,\cq 7/\Q \rangle = \sigma_5$ when restricted to $K=\Q(\sqrt{-7})$ equals $\langle 5, K/\Q \rangle$, and thus, is not identity on $K$ since $5\fO_K$ is a prime ideal in $K$. This follows from Kummer's Theorem (cf. \cite[\S1.Thm 7.4]{Janusz}) as $x^2 + 7$ is irreducible modulo $5$; in other words, $-7$ is a quadratic non-residue modulo $5$ and this is captured by $\chi_{-7}(5) = -1$.  
\end{quote}

Hence, in order to get $\sigma_b(\sqrt{21}) = - \sqrt{21}$, we need $\sigma_b(\sqrt{-3}) = \sqrt{-3}$. This implies that $b \equiv 1 \mod 3$, hence $p \equiv 1 \mod 3$, and $p$ splits completely in $\cq 3$. As a result, the sieve cannot be used since we could not choose $b$ so that $(b-1,21)=1$. Therefore, we have to exclude cases where $\ff_2=3(f,\ff_2)$ and $\chi_{-\fd_2/3}(a) = -1$ when $\ff_2 \nmid f$ (see Lemma \ref{L:K2quad}).

\subsection{Conditional Results}
Next, we move onto the asymptotic results similar to Serre's Theorem in \eqref{SerreAsym}. We first introduce a few facts and give some definitions. 

For each integer $m \ge 1$, there exists a representation
\eqs{
\rho_m = \rho_{E/\Q,m} : G_\Q=\Gal(\overline \Q/\Q) \longrightarrow \text{Aut}(E[m]) \simeq \text{GL}_2(\Z/m\Z)}
determined by the action of the absolute Galois group $G_\Q$ on the torsion group $E[m]$. The fixed field of its kernel is the $m$-division field $K_m$, so
\eqn{\label{GalKm}\gaq{K_m} \simeq \rho_{m}(G_\Q).}
In 1972, Serre proved (cf.~\cite{Serre72}) that
\eqs{S_E =\{ p \text{ prime}: \rho_p(G_\Q) \neq \text{GL}_2(\Z/p\Z) \}
}
is finite if and only if $E$ is non-CM. In this case, the \emph{Serre constant} of $E/\Q$ is defined as the number
\eqn{\label{A(E)}
A(E) = 30 \pr{p>5\\p \in S_E} p.}
Furthermore, we define the constant
\eqs{\me = \prod_{p \mid A(E)N_E} p.}
 
We shall denote our prime counting function by 
\eqs{\pE =\#\{p\le x : p\nmid 2N_E, p \equiv a \mod f, \text{ and } \E(\F_p) \text{ is cyclic} \}.}

Arithmetic functions $\omega, \tau$, $\sigma$, and $H$  that appear below are 
\eqn{\label{E:H}
\omega(n) = \sum_{p\mid n} 1, \quad \tau(n) = \sum_{d>0, d \mid n} 1, \quad \sigma(n) = \sum_{d>0, d \mid n} d, \quad H(n) = \sum_{d \mid n} \su{1 \le k \le d \\ d \mid k^2} 1,}
and, as usual, $\varphi$ is Euler's totient function. 
\begin{theorem} \label{T:NonCMasymp}
Let $E/\Q$ be a non-CM curve. Assuming GRH holds for all Dedekind zeta functions of the fields $K_d\cq f$ for all square-free $d\ge 1$, we have
\eqs{\pE = \delta_E(f,a) \li(x) + E(x),}
where 
\eqn{\label{densityDeltaEaf}
\delta_E(f,a):=\sum_{d=1}^{\infty}\frac{\mu(d)\gamma_{a,f}(K_d)}{[K_d\cq f:\Q]},
}
where $\mu$ denotes the M\"obius function, and $\gamma_{a,f}(K_d) = 1$ if $\sigma_a$ fixes $K_d \cap \cq f$, and is $0$ otherwise, and  the error term $E(x)$ satisfies
\bsc{\label{E:nonCMError} 
E(x) &\ll x^{1/2}f\log(fxN_E) + x^{5/6} \( \frac{H(f) \log^2 (fxN_E)}{f} \)^{1/3} \\
&\quad + x^{5/8} \( \frac{\tau(f_2) \me^3 \log^3 (fxN_E)}{\varphi(f)\log x} \)^{1/4} + \frac{\tau(f_2)\me^3}{x^{1/2}\varphi(f)\log x}. 
}
Here, $f_2$ denotes the largest divisor of $f$ that is coprime to $\me$.
\end{theorem}
\begin{rem}
It follows from \eqref{eq:Hbound} that $H(n)$ satisfies
\eqn{\label{E:bound4H}
2^k \sigma\Bigl( \prod_{i \le k} p_i^{\ceil{\alpha_i/2}-1} \Bigr) \le H\Bigl(\prod_{i\le k} p_i^{\alpha_i} \Bigr) \le 2^k \sigma\Bigl( \prod_{i\le k} p_i^{\fl{\alpha_i/2}} \Bigr).}
In particular, for $f= \prod_{i\le k} p_i^{\alpha_i}$, it follows from \cite{Ivic} that
\eqs{H(f) < 2.59\cdot 2^k \sqrt f \log \log \sqrt f,}
whenever $\prod_i p^{\fl{\alpha_i/2}} \ge 7$, and $H(f) < 2^{k+1} \sqrt f$ otherwise. The last  inequality, of course, gives only a crude estimate since the behavior of $H$ is not very regular. For example, if $f$ is a large prime, then $H(f) = 2$ while $H(f^2) = 2+f > f$.
\end{rem}

In this paper, we did not try to see if a weaker quasi-GRH would work as in \cite{Co1}, but rather wanted to get explicit and smaller error terms that can be obtained under GRH. 

As for the positivity of the density, we have the following.
\begin{theorem} \label{T:DensitynonCM}
Let $E/\Q$ be a non-CM curve. If $(f,\me)=1$, and $K_2 \neq \Q$, then 
\multn{\label{E:densitynonCM}	
\delta_E (f,a) \ge \frac 1 {\varphi(f)} \prod_{\substack{p \nmid \me \\(p,f) \mid a-1}} \( 1 - \frac{\varphi(p,f)}{[K_p:\Q]}\)  \prod_{2< p\mid \me} \( 1 - \frac{1}{p-1}\) \\
\cdot \frac 1 {[K_2:\Q]} \( [K_2:\Q]- 1  - \frac{\mu(\ff_2)([\ab 2:\Q] -1)} {\prod_{2 < p \mid \ff_2} ( p - 2 )} \) >0, 
}
where $\varphi(p,f)$ stands for $\varphi(\gcd(p,f))$.	
\end{theorem}
\begin{rem} \label{CoMuCorrection}
Note that when $f=1$, \eqref{E:densitynonCM} would imply $\delta_E$ in \eqref{DeltaE} is at least
\eqs{\frac 1 2 \biggl( 1 - \frac {\mu(\ff_2)} {\prod_{2 < p \mid \ff_2} ( p - 2 )} \biggr) \prod_{2< p\mid \me} \( 1 - \frac{1}{p-1}\) \prod_{p \nmid \me} \( 1 - \frac{1}{[K_p:\Q]}\).}
This is obtained in the same way as Cojocaru and Murty had their result on page 621 of \cite{CoMu}, yet the results are different. The reason is that when $\ff_2$ is not a prime, then $\ab 2$ may have non-trivial intersections with $\cq d$ with square-free $d \mid \me$, even though $\ab 2 \cap \cq q =\Q$ for each prime $q \mid d$. They seem to have overlooked this point in their work.
\end{rem}
By the definition of $\me$, we have $[K_p:\Q] = (p^2-p)(p^2-1) \asymp p^4$ for $p \nmid M_E$. Thus, we obtain from \eqref{E:densitynonCM} that
\eqs{\delta_E(f,a) \gg \frac 1 {\varphi(f)} \prod_{2< p \mid \me} \( 1 - \frac 1 {p-1}\) = \frac{2\varphi(\me)}{\varphi(f)\me} \gg \frac 1 {\varphi(f)\log \log \me}.}

Next, we mention another result known about $\delta_E (f,a)$. In 2010, Nathan Jones proved (cf.~\cite{Jones}) that almost all non-CM curves are Serre curves, and in 2015, Julio Brau Avila showed in his thesis (cf.~\cite{Brau}) that $\delta_E(f,a)$ is positive for Serre curves. Although an asymptotic formula is not given in Brau's work, the density is computed explicitly using a different approach. 

Brau also considers the curve 
\eqs{y^2 = x^3 + x^2 + 4x + 4,} 
as an example, which is a non-CM and non-Serre curve with $K_2 = \cq 4$ (so $\ff_2=4$), $N_E = 20$, and $A(E)=30$ (yielding $M_E= 30$). Proposition 2.5.12 in \cite{Brau} then states that $\delta_E(f,a)=0$ for this curve if and only if $4 \mid f$ and $a \equiv 1 \mod 4$. Proposition \ref{ConjEZpart} and Theorem \ref{T:SieveK2quad} in this paper show that there are infinitely many primes $p \equiv a \mod f$ for which $\E(\F_p)$ is cyclic unless $4 \mid f$ and $a \equiv 1 \mod 4$, in which case there are at most finitely many such primes, which agrees with Brau's result, and our result is unconditional.

For Serre curves, the intersection of division fields is much better understood, which makes them easier to study. In particular, all $K_p$ are non-abelian. However, for an arbitrary non-CM elliptic curve, things are more complicated. As is apparent from \eqref{densityDeltaEaf}, in order to study $\delta_E(f,a)$, we have to understand $\gamma_{a,f}(K_d)$ and the intersections $K_d \cap \cq f$. For non-CM curves, Lemma \ref{DivCycIntersection} describes the latter for certain values of $d$, thanks to which we can then write $\delta_E (f,a)$ as a product of two factors; an infinite convergent product, and the finite sum (see Lemma \ref{densitysumprod})
\eqs{\sum_{d \mid \me} \frac{\mu(d)\gamma _{a,g}(L_d)}{[L_d:L_d\cap \cq g]},}
where $g$ is the largest divisor of $f$ coprime to $\me$, $L_p/\Q$ are Galois extensions lying inside $K_p$ and must be chosen appropriately. In particular, $L_2$ has to be either $K_2$ or $\ab 2$, and the main difficulty here is to understand the intersections $L_d \cap \cq g$, and the constants $\gamma_{a,g}(L_d)$ for every $d \mid \me$ for a general modulus $g$ and residue class $a$. 
Since we do not have enough information, we cannot write this sum as a product. 

Brau's result in \cite{Brau} for Serre curves which we record below and his example mentioned above both support our prediction.
\begin{theorem}[{\cite[Corollary 2.5.9]{Brau}}]
For a Serre curve $E/\Q$, $\delta_E(f,a)$ is positive for coprime positive integers $a$ and $f$.
\end{theorem}
Next, we turn to CM curves. We assume as in \cite{Akbary} and \cite{CoMu} that the endomorphism ring is isomorphic to the full ring of integers. The exact definition of the arithmetic function $G_D(a,f)$ that appears inside the error term below is given in the proof.
\begin{theorem}\label{T:CMasymp}
Let $E/\Q$ be an elliptic curve with End$_{\overline{\Q}}(E) \simeq \fO_K$, where $\fO_K$ is the ring of algebraic integers of an imaginary quadratic field $K=\Q(\sqrt{-D})$.  If GRH holds for all Dedekind zeta functions of the fields $K_d\cq f$ for all square-free $d\ge 1$, then
\eqs{\pE = \delta_E(f,a) \li(x) + E(x),}
where $\delta_E(f,a)$ is given by \eqref{densityDeltaEaf} and the error term $E(x)$ satisfies
\multn{\label{CMError}
E(x) \ll x^{3/4} \( \frac{\log(fxN_E)}{\log x}\)^{1/2} + x^{3/4} \( \frac{\log(fxN_E) G_D(a,f)}{f^3} \)^{1/2}\\
\qquad + x^{1/2}f\log(fxN_E) +  x^{1/2}\(\frac 1 f + \frac{\log x}{f^2} \)G_D(a,f).}
Here, $G_D(a,f)$ is the cardinality of the set given by \eqref{Gaf}, is multiplicative in the second variable and satisfies 
\eqn{\label{Gbound4f}
G_D(a,f) < c \cdot  4^{\omega(f)} \tau(f) f^2 }
where $c=2$ if $D \equiv 1,2 \mod 4$, or $D \equiv 3 \mod 4$ and $f$ is odd, and $c=49$ otherwise. 
\end{theorem}

As for the density, we have the following result.
\begin{theorem} \label{T:DensityCM}
The density $\delta_E(f,a)$ in Theorem \ref{T:CMasymp} is positive if one of the following holds:
\begin{enumerate}
\item $K_2 \cap K = \Q$, and $\gamma_{a,f}(K_2K) = \gamma_{a,f}(K_2)\gamma_{a,f}(K)$, and $K_2 \subsetneq \cq f$ or $\sigma_a$ does not fix $K_2\cap \cq f$, and $K \subsetneq \cq f$ or $\sigma_a$ does not fix $K\cap \cq f$, 

\item $\ab 2 = K$, and $K_2 \subsetneq \cq f$ or $\sigma_a$ does not fix $K_2 \cap \cq f$.
\end{enumerate}
\end{theorem}
\begin{rem}
We did not attempt to handle the CM case without GRH in this paper even though division fields are better understood for these curves, and one may be able to improve Theorems \ref{T:CMasymp} and \ref{T:DensityCM}. We leave this task to a seperate paper.
\end{rem}

\noindent{\bf Acknowledgements.} We would like to thank Ram Murty for helpful discussions, and suggesting the use of linear sieve that plays a significant role and makes an important contribution to the results we obtained. We also wish to thank Ernst Kani for sharing his notes given in appendix that are used in several parts of the paper and play an essential role in the proof of some of the Theorems.
\section{Proofs of Unconditional Results}
\subsection{The Linear Sieve} %\label{SieveSec}
Assume that $F\ge 1$ is an integer satisfying 
\eqn{\label{sizeofmodulus}
F \ll (\log x)^A \quad \text{for some } A\ge 0,}
$c$ is an integer coprime to $F$ such that $(c-1,F)$ has no odd prime divisors. Put
\eqs{\cA = \{ p-1 : p\le x, p \equiv c \mod F \}}
and, as usual, define 
\eqs{\cP(z) = \prod_{q < z, q \in \cP} q,}
where $\cP$ is the set of \emph{odd primes coprime to} $F$. We seek a lower bound for
\bs{S(\cA,\cP,z) &= |\{ n \in \cA: (n,\cP(z))=1\}|. 
}
For $d \mid \cP(z)$, we have
\eqs{\cA_d := \su{n\in \cA\\d \mid n}1 =  \pi(x;dF,c_d) = \frac{\omega(d)}{d} \frac {\li(x)}{\varphi(F)} - r(\cA,d),}
say. Here, $\pi(x;dF,c_d)$ denotes the number of primes $p \le x$ that are congruent to $c_d$ modulo $dF$, $c_d$ is the unique integer (by Chinese Remainder Theorem) modulo $dF$ satisfying $c_d \equiv 1 \mod d$ and $c_d \equiv c \mod F$, and $\omega(d) = d/\varphi(d)$ satisfies $0 < \omega(q) < q$ for all odd primes $q$. Furthermore, the inequalities  
\bs{
\pr{w \le p < z\\p \nmid 2F} \( 1- \frac{\omega(p)} p\)^{-1} 
&< \exp\Biggl(\, \su{p \ge w\\ p > 2} \frac 1{p^2-2p} \Biggr) \prod_{w \le p \le z} \( 1- \frac 1 p\)^{-1}  \\
&\le \frac{\log z}{\log w} \( 1 + \frac K {\log w} \)}
and
\eqs{\su{w \le p < z\\ p \in \cP} \sum_{k \ge 2} \frac{\omega(p^k)}{p^k} = \su{w \le p < z\\ p \in \cP} \frac 1 {(p-1)^2} \le \frac{L}{\log (3w)}}
hold for all $z > w \ge 2$ for some constants $K, L > 1$, where in the second inequality of the first equation we use Merten's estimate
\eqs{\prod_{p \le x} \Bigl(1-\frac 1 p\Bigr) ^{-1} = e^\gamma \log x + O(1).}
We have verified so far that the necessary conditions given in \cite{Iwaniec} by equations (1) and (2) are satisfied. Hence, we are now ready to use the lower bound sieve of Iwaniec in \cite{Iwaniec}. Thus, assume that $\eps_1 \in (0,1/3)$, and $2 \le y^{1/4} \le z < y^{1/2}$. Then, it follows from \cite[Thm 1]{Iwaniec} that 
\eqs{S(\cA, \cP, z) \ge \frac{\li(x)}{\varphi(F)}\pr{2 < p < z\\ p \nmid F} \( 1 - \frac 1 {p-1} \) \big\{ f(s) - E(\eps_1, y, K,L)\big\} - R(\cA, y),}
where $s = \log y/\log z$, $ E(\eps_1, y, K,L) \ll \eps_1 + \eps_1^{-8} e^{K+L} (\log y)^{-1/3}$ and 
\eqs{R(\cA, y)= \sum_{l < \exp(8/\eps_1^3)} \su{d < y\\d \mid \cP(z)} \lambda_l(d) \( \frac{\li(x)}{\varphi(dF)}-\pi(x;dF,c_d) \)}
for some well factorable functions $\lambda_l$ (see the paragraph before \cite[Lemma 2]{Heath} for the definition). Here, the implied constant is absolute. The function $f(s)$ that appears above is a continuous solution of a system of differential-difference equations given in \cite{Iwaniec}, and in the interval $2 \le s \le 4$ that we are interested in $f(s)$ is given by (cf. \cite[p.~126]{Greaves})
\eqs{f(s) = \frac{2e^\gamma}{s}\log (s-1),}
where $\gamma = 0.5772156649\ldots$ is the Euler-Mascheroni constant. 

Now, we choose $y=x^{4/7-\eps_2}$ and $z = y^{1/(2+\eps_2)}$ with a fixed $\eps_2\in (0,1)$ so that $s = 2 + \eps_2$, and
\eqs{f(s) > \frac{\eps_2 e^\gamma}{2+\eps_2} > \eps_2/2.} 
For $\eps_1$ sufficiently small in terms of $\eps_2$ and $x$ sufficiently large, we get
\eqs{f(s) - E(\eps_1,y,K,L) > \eps_2/3.}
Furthermore, it follows from \cite[Lemma 2]{Heath} that for a given $\eps_2$ and any $B>0$, 
\eqs{R(\cA, y) \ll x F^k (\log x)^{-B},}
for some fixed positive integer $k$, where the implied constant may depend on $c, \eps_2$, and $B$. Then, choosing $B=(k+1)A+3$, it follows from \eqref{sizeofmodulus} that
\eqs{S(\cA, \cP, z) \ge c(\eps_2,A) \frac{x}{(\log x)^{2+A}}}
for sufficiently large $x$. For $\eps_2\in (0,2/35)$, we see that $z = x^\alpha$ with 
\eqs{\alpha =\alpha(\eps_2) = \frac{4/7-\eps_2}{2+\eps_2} = \frac 1 4 + \frac{2/7-5\eps_2}{8+4\eps_2}> \frac 1 4.} Furthermore, since
\eqs{\sum_{q \ge x^{\alpha}} \su{p \le x\\ q^2 \mid p-1} 1 < \sum_{x^{\alpha} \le q < \sqrt x} \( \frac{x}{q^2}+ 1\) \ll x^{1 - \alpha} = o\(\frac{x}{\log^{2+A} x}\),
}
we can also assume that each $p-1$ counted in $S(\cA, \cP,x^\alpha)$ has \emph{distinct} odd prime divisors $q \ge x^\alpha$ coprime to $F$. Finally, since there are only finitely many divisors of $N_E$, we obtain the following result:
\begin{lemma}  \label{Sieve}
Let $A\ge 0$ and $\eps \in (0,2/35)$ be given. Assume that $c$ and $F$ are positive coprime integers such that $F \ll (\log x)^A$ and no odd prime divides $(c-1,F)$. Then, there is some $\alpha=\alpha(\eps)>1/4$ and a positive constant $c(\alpha,A)$ such that for $x$ sufficiently large, there are at least $c(\alpha,A) x/(\log  x)^{2+A}$ primes $p \le x$ with $p\equiv c \mod F$ and $p \nmid N_E$ such that odd prime divisors $q$ of $p-1$ are distinct, coprime to $F$ and satisfy $q \ge x^\alpha$.
\end{lemma}

\subsection{Proofs of Theorems \ref{T:SieveK2cubic} and \ref{T:SieveK2quad}}
As mentioned in the introduction, Murty and Gupta showed in \cite{GM} unconditionally that for any elliptic curve $E/\Q$ for which $K_2 \neq \Q$, there are infinitely many primes $p$ for which $\E(\F_p)$ is cyclic. The first step in their proof is to make sure $p$ does not split completely in $K_2$, which is established by imposing a congruence condition on $p$ as mentioned in \cite[Lemma 3]{GM}. Since this result plays a fundamental role in this paper and since they do not give any details, we show below that there is in fact an appropriate arithmetic progression that serves this purpose.
\begin{lemma}\label{Ram-Gupta}
If $K_2 \neq \Q$, there exists some $b \in (\Z/\ff_2\Z)^\times$ such that $\gamma_{b,\ff_2} (K_2) = 0$ and the odd part of $\ff_2$ is coprime to $b-1$. 
\end{lemma}
\begin{rem} \label{R:GM}
As mentioned in the introduction, to be able to apply the linear sieve, it is of fundamental importance to make sure that no odd prime divides $(\ff_2,b-1)$, and that is exactly why we need to prove that there is at least one such $b$. Otherwise, only finding some $b\in (\Z/\ff_2\Z)^\times$ such that $\gamma_{b,\ff_2} (K_2) = 0$ can easily be accomplished by choosing an automorphism $\sigma \in \gaq{\cq{\ff_2}} \setminus \Gal\bigl(\cq{\ff_2}/(K_2 \cap \cq{\ff_2})\bigr)$.
\end{rem}
\begin{proof}
Note that $K_2 \cap \cq{\ff_2} = \ab 2$.

Assume first that $[\ab 2:\Q]=2$. Then, $\ab 2 = \Q(\sqrt D)$ for some square-free integer $D$, and  
\eqn{\label{QuadConductor}
\ff_2 = \left\{\def\arraystretch{1.2} \begin{array}{r@{\quad\text{if }}l}
4|D| & D \equiv 2,3 \mod 4,\\ 
|D| & D \equiv 1 \mod 4,
\end{array}
\right.}
is the absolute value of the discriminant $\fd_2$ of $\ab 2$ over $\Q$ (cf.~\cite[Corollary~VI.1.3]{Janusz}). We choose $b =3$ if $D=-1, 2$; $b=7$ if $D=-2$. For $|D|>2$, let $p$ be the smallest \emph{odd} prime divisor of $D$, and choose $b$ as the unique solution modulo $\ff_2$ of the system of congruences
\eqs{%\arraycolsep=1.4pt
\def\arraystretch{1.1}
\begin{array}{l@{\qquad\text{if }}l}
\Bigg\{ 
\begin{array}{ll}
b \equiv g_p &\mod p \\
b \equiv g_q^2 &\mod q \qquad (\forall q \mid D/p)\\
b \equiv 1 &\mod 4
\end{array} & D \equiv 3 \mod 4 \\[5mm]
\bigg\{ 
\begin{array}{ll}
b \equiv g_p &\mod p \\
b \equiv g_q^2 &\mod q \qquad (\forall q \mid D/p)
\end{array}
& D \equiv 1 \mod 4 \\[5mm]
\Bigg\{ 
\begin{array}{ll}
b \equiv g_p &\mod p \\
b \equiv g_q^2 &\mod q \qquad (\forall q \mid D/(2p))\\
b \equiv 1 &\mod 8
\end{array} & D \equiv 2 \mod 4
\end{array} 
}
Here, $g_p$ denotes a primitive root modulo $p$ for each odd prime divisor of $D$. Since $q > 3$ for any $q \neq p$, $g_q^2 \not\equiv 1 \mod q$. Furthermore, $\sigma_b(\sqrt D) = - \sqrt D$. Thus, we have the desired $b$. 

Next, assume that $[K_2:\Q]=3$ (note $K_2 =\ab 2$). Hasse proved (cf.~\cite{Hasse}) that  
\eqn{\label{CubicConductor}
\ff_2 = p_1 p_2 \cdots p_r,} 
where $p_1, \ldots, p_r$ are either all distinct primes with $p_i \equiv 1 \mod 3$, or all except one, say $p_r$, are such primes, and $p_r=9$. 

If $r=1$, any $b$ which is not a cube modulo $p_1$ works. In particular, there are $2\varphi(p_1)/3$ choices for $b$. If $r>1$, write $\ff_2=p_1 m$. Since $K_2 \cap \cq n = \Q$ for any $n \mid m$ (otherwise, $K_2 \subset \cq m$), we have 
\eqs{\gaq{\cq m K_2} \simeq \gaq{K_2} \times \gaq {\cq{p_2}} \times \cdots \times \gaq{\cq{p_r}}. }
Thus, there are $2\prod_{i=2}^r (\varphi(p_i)-1)$ choices for an automorphism $\tau \in \gaq{\cq m K_2}$, which is not identity on $K_2$ and on any $\cq{p_i}$ for $i=2, \ldots, r$. Furthermore,  
\eqs{[\cq{\ff_2}:\Q] = \frac{[\cq m K_2 :\Q][\cq{p_1}:\Q]}{[L:\Q]}= \frac{3\varphi(\ff_2)}{[L:\Q]},}
where $L = \cq m K_2 \cap \cq{p_1}$, implies $[L:\Q] = 3$. Since $[\cq{p_1}:\Q] > 3$, we can extend $\tau_{|L}$ to a non-identity automorphism $\beta$ of $\gaq{\cq{p_1}}$. Since $\tau$ and $\beta$ agree on $L$, it follows from Galois theory that there is a $\sigma \in \Gal(\cq{\ff_2}/\Q)$ which extends $\tau$ and $\beta$. Then, $\sigma$ uniquely determines some $b \in (\Z/\ff_2\Z)^\times$ such that $(b-1,\ff_2)=1$ and $\gamma_{b,\ff_2}(K_2)=0$ as desired.
\end{proof}

\begin{rem} \label{R:Kronecker}
Let $\chi_{\fd_2}$ be the real primitive character of conductor $\ff_2$ given by the Kronecker symbol $(\tfrac{\fd_2}{\cdot})$. Then, $\gamma_{b,\ff_2} (K_2) = 1$ if and only if $b \in \ker \chi_{\fd_2}$ (to see how this character plays a role, see for example, \cite[I.7.4 and  pp.~250-1]{Janusz}). So, when $[\ab 2:\Q]=2$, we choose $b$ in such a way that $b \not\in \ker \chi_{\fd_2}$ and that $ b \not\equiv 1 \mod q$ for odd $q \mid D$.
\end{rem}

The next result is needed in the proof of Theorem \ref{T:SieveK2cubic}.
\begin{lemma} \label{L:K2cubic}
Assume that $[K_2 :\Q]=3$. Let $m>1$ be a proper divisor of $\ff_2$ and $a$ an integer such that $(m,a(a-1))=1$. Then, there is some $b$ satisfying conditions of Lemma \ref{Ram-Gupta} such that $b \equiv a \mod m$.
\end{lemma}
\begin{proof}
%If $\omega(\ff_2)=1$, then $\ff_2=9$, $m=3$ and $a=2$. In this case, we take $b=a$. Assume now that $\omega(\ff_2)>1$ and w
Write $\ff_2 = pdm=pn$, where $p$ is a prime, $d \ge 1$ and $(d,m)=1$. Since $K_2 \cap \cq n = \Q$, there is some $\tau \in \gaq{\cq n K_2}$ which is not identity on $K_2$ and on $\cq q$ for each prime (if any) $q \mid d$, while it equals $\sigma_a$ on $\cq m$. If $p=3$, then $\cq n K_2 = \cq{\ff_2}$. Thus, $\tau=\sigma_b$ for some $b$. If $3 \mid m$, then $b \equiv a \mod m$ implies $b \equiv 2 \mod 3$ since $(m,a(a-1))=1$. Otherwise, $3 \mid d$ and $\sigma_b \neq 1_{\cq 3}$ implies $b \equiv 2 \mod 3$. In either case, we obtain the desired result. If $p \neq 3$, we put $L=\cq{n} K_2 \cap \cq p$. Then, 
\eqs{[L:\Q] = \frac{[\cq{n}:\Q][K_2:\Q][\cq p:\Q]}{[\cq{\ff_2}:\Q]} = \frac{3\varphi(\ff_2/p)\varphi(p)}{\varphi(\ff_2)} = 3 < \varphi(p).}
Thus, we can extend $\tau_{|L}$ to a non-identity automorphism $\beta$ of $\cq p$. Since $\tau$ and $\beta$ agree on $L$, it follows from Galois theory that there is a $\sigma_b \in \Gal(\cq{\ff_2}/\Q)$ which extends $\tau$ and $\beta$ for some $b$ with the desired property.
\end{proof}

\begin{proof}[Proof of Theorem \ref{T:SieveK2cubic}]
If $\ff_2 \nmid f$, then we can write $f=mg$ with $m=(\ff_2,f) < \ff_2$. Applying Lemma \ref{Ram-Gupta} if $m=1$, and Lemma \ref{L:K2cubic} for $m>1$ yields some $b$ with which the system $p \equiv b \mod \ff_2$ and $p \equiv a \mod f$ is solvable since $m \mid a-b$, and there is a unique solution, say, $c$ modulo $F=[f,\ff_2]$. Applying Lemma \ref{Sieve} to primes $p \equiv c \mod F$, we find some $\alpha > 1/4$ and a set of primes $S_\alpha(x)$ having properties stated in Lemma \ref{Sieve}. We would like to show that the number of $p \in S_\alpha(x)$ for which $\E(\F_p)$ is not cyclic is negligible. The rest of the proof follows the proof of \cite[Theorem 1]{GM}, but we shall include it here. 

Recall that $|\E(\F_p)| = p + 1 - a_p$, where $a_p$ denotes the trace of the Frobenius associated to E and $p$. Put
\eqs{S(b, x) = \{p \in S_\alpha(x) : a_p = b\}.}
By Hasse's inequality, $S_\alpha(x)$ is the union of $S(b, x)$ with $|b| \le 2\sqrt x$. Take a prime $p \in S(b, x)$ for which $\E(\F_p)$ is not cyclic. Then, $p$ splits completely in $K_q$, for some odd prime $q$. Since $\cq q \subset K_q$, $q \mid p-1$ and the fact that $p \in S_\alpha(x)$ implies $q \ge x^\alpha$ and is coprime to $[f,\ff_2]$. Moreover, $q^2 \mid |\E(\F_p)| = p + 1 - a_p = p-1 +(2-b)$, thus $q \mid b - 2$. Notice that $b \neq 2$ since odd prime divisors of $p - 1$ are distinct. Since $q \ge x^\alpha$ with $\alpha > 1/4$ and 
$|a_p- 2| \ll x^{1/2}$, there is only one such prime $q$ for a given $b$, for $x$ sufficiently large. Therefore, any $p \in S(b,x)$ for which $\E(\F_p)$ is not cyclic satisfies
\eqs{p \equiv b-1 \mod q^2}
and the number of such $p$ is $< x/q^2 + O(1) \ll x^{1-2\alpha}$. The total number of $p \in S_\alpha(x)$ for which $\E(\F_p)$ is not cyclic is, therefore, $\ll x^{3/2 - 2\alpha} = o(x/(\log x)^{2+A})$.

If $\ff_2 \mid f$ and $\gamma_{a,f}(\ab 2) = 0$, we can apply Lemma \ref{Sieve} with the pair $(a,f)$, and repeat the same arguments above.
\end{proof}
\begin{lemma}\label{L:K2quad}
Assume that $[\ab 2:\Q]=2$, $m>1$ is a proper divisor of $\ff_2$, $(a,m)=1$ and the odd part of $m$ is coprime to $a-1$. Then, there is some $b$ satisfying conditions of Lemma \ref{Ram-Gupta} such that $b \equiv a \mod m$ unless $\ff_2=3m$ and $\chi_{-\fd_2/3}(a) = -1$.
\end{lemma}
\begin{proof}
By remark \ref{R:Kronecker}, we need to find some $b$ with $(b,\ff_2)=1$ such that $\chi_{\fd_2}(b)=-1$ and that $ b \not\equiv 1 \mod q$ for odd $q \mid D$. Write $\ff_2 = pdm =pn$ with $d \ge 1$. Whenever $p=3$, we need to choose $b \equiv 2 \mod 3$ so that $3 \nmid b-1$, and $b \equiv a \mod m$. This gives $\chi_{\fd_2}(b) = (\tfrac{b}{3})\chi_{-\fd_2/3}(b) = - \chi_{-\fd_2/3}(b)$. If $d=1$, this implies $\chi_{-\fd_2/3}(a)$ should be $1$ since otherwise $\gamma_{b,\ff_2}(\ab 2) = 1$. If $d \neq 1$ and $(d,m)=1$, we choose $b$ modulo $d$ in such a way that $q \nmid b-1$ for each odd $q \mid d$ and that $\chi_{\fd_2} (b) = -1$. This can be done since odd prime divisors of $d$ are larger than $3$. If $(d,m)\neq 1$, it equals 4 or 8. In this case, we choose $b$ similarly for odd prime divisors of $d$, and congruent to $a$ modulo the odd part of $m$. We finally choose $b$ modulo $(d,m)$ so that $\chi_{\fd_2}(b)=-1$. If $3 \nmid \ff_2$, then we choose $b$ similarly.
\end{proof}

\begin{proof}[Proof of Theorem \ref{T:SieveK2quad}]
If $\ff_2 \nmid f$, then we can write $f=mg$ with $m=(\ff_2,f) < \ff_2$. Applying Lemma \ref{Ram-Gupta} if $m=1$, and Lemma \ref{L:K2quad} for $m>1$ yields some $b$ with which the system $p \equiv b \mod \ff_2$ and $p \equiv a \mod f$ is solvable since $m \mid a-b$, and there is a unique solution modulo $[f,\ff_2]$. Applying Lemma \ref{Sieve} and proceeding as in the proof of Theorem \ref{T:SieveK2cubic}, we get the result. If $\ff_2 \mid f$ and $\gamma_{a,f}(\ab 2) = 0$, we can apply Lemma \ref{Sieve} with the pair $(a,f)$. 
\end{proof}

\section{Proofs of Theorems \ref{T:NonCMasymp} and \ref{T:DensitynonCM}}

Throughout this section we assume that $E$ is an elliptic curve over $\Q$ that has no complex multiplication. 

\subsection{Preliminaries}
Recall that $\ff_n$ is the conductor of $\ab n$. It follows from \cite[V Thm 1.10, p.324]{Neukirch} that $\ff_n$ consists of primes that ramify in $\ab n$. Also, primes that ramify in $K_n$ are the divisors of $nN_E$ (see, for example, \cite[p.179]{Sil86}). Since these primes also ramify in $K_n$, $\ff_n \mid (nN_E)^\infty$. In particular, $\ff_2 \mid \me^\infty$ and we use this implicitly in the proof of Theorem \ref{T:DensitynonCM}.

\begin{lemma}[{\cite[Lemma 2.1]{CoMu}}] \label{cyclicity} 
Let $E$ be an elliptic curve defined over $\Q$, and $p$ a prime with $p \nmid N_E$. Then, for any prime $q \neq p$, $\E(\F_p)$ contains a subgroup isomorphic to $\Z/q\Z\times \Z/q\Z$ if and only if $p$ splits completely in $K_q$. Therefore,  for odd $p$, $\E(\F_p)$ is cyclic if and only if $p$ does not split completely in $K_q$ for any prime $q\neq p$.
\end{lemma}
\begin{lemma} \label{compofdivfields}
If $(d,e) =1$, then $K_{de} = K_d K_e$.	
\end{lemma}
\begin{proof}
Since $K_d, K_e \subseteq K_{de}$, $K_dK_e \subseteq K_{de}$. Now, take any $de$-torsion point $(x,y)$ of $E$, and note that since $(d,e)=1$,  $(x,y)=ad(x,y)\oplus be(x,y)$ for some integers $a$ and $b$, where $\oplus$ denotes the group operation on $E$; that is, $(x,y)$ is the sum of a $d$-torsion and an $e$-torsion point. Thus, the claim follows. 
\end{proof}

\begin{lemma} \label{DivCycIntersection}
If $(e, A(E))=1$, then $K_e \cap \cq g =\cq{(e,g)}$, where $A(E)$ is Serre's constant defined in \eqref{A(E)}.
\end{lemma}
\begin{proof}
By \cite[Appendix Cor.~13]{Co3}, $\cq e$ is the maximal abelian extension of $\Q$ in $K_e$. Thus, $K_e \cap \cq g$, being abelian, lies in both $\cq e$ and $\cq g$, and also contains their intersection since $\cq e \subseteq K_e$. 
\end{proof}
\begin{lemma}[Theorem \ref{A:mainthm} in Appendix]\label{L:disjointness}
If $(m,n\me)=1$, then $K_n \cap K_m = \Q$.
\end{lemma}

Below we give an effective version of Chebotarev's Density Theorem.
\begin{lemma}[{\cite[Thm 3.1 + Lemma 3.4]{CoMu}}] \label{effectiveChebo}
Let $L/\Q$ be a Galois extension of discriminant $\Delta_L$, $G= \gaq{L}$, $C\subseteq G$ a conjugacy class, and $\cP(L)$ the set of primes $p$ that ramify in $L$. Then, assuming GRH for the Dedekind zeta function of $L$, 
\eqs{
\pi_C(x,L/\Q)=\frac{|C|}{|G|}\li(x)+ O\biggl(x^{1/2} \log \Bigl( x [L:\Q] \prod_{p \in \cP(L)} p \Bigr) \biggr),
}
where
\eqs{\pi_C(x,L/\Q) = |\{ p \le x: p \nmid \Delta_L, \mathrm{Frob}_p(L/\Q) \subseteq C\}|.}
\end{lemma}
\begin{lemma} \label{NonCMPhiSum}
For real $Y \ge 1$ and integer $k \ge 1$, 
\eqs{\sum_{n>Y}  \frac 1 {n^k \varphi(n)} \ll Y^{-k}.}
\end{lemma}
\begin{proof}
We have
\bs{
\sum_{Y < e \le Z}  \frac 1 {e^k \varphi(e)} 
&= \sum_{Y < e \le Z}  \frac 1{e^{k+1}} \prod_{p\mid e} \frac{p}{p-1} \\
&< \prod_p \( 1+ \frac 1{p^2-1}\) \sum_{Y < e \le Z}  \frac 1{e^{k+1}} \sum_{d \mid e} \frac{\mu(d)^2}{d}\\
&< e^{\pi^2/6} \sum_{Y < ed \le Z}  \frac 1{e^{k+1}d^{k+2}} \\
&\ll \sum_{d \le Z}  \frac 1{d^{k+2}} \sum_{e> Y/d}  \frac 1 {e^{k+1}} \ll  Y^{-k} \sum_{d \le Z}  \frac 1{d^2},
}
and taking limit as $Z \to \infty$, the result follows.
\end{proof}
\begin{lemma} \label{CMPhiSum}
For $Y>1$, 
\eqs{\sum_{n>Y} \frac 1 {\varphi(n)^2} \ll \frac 1 Y.}
\end{lemma}
\begin{proof}
Note that for any $x \ge 1$,
\eqs{\fl x \le \sum_{n \le x} \frac n {\varphi(n)} = \sum_{d \le x}\frac{\mu(d)^2}{\varphi(d)} \sum_{n \le x/d} 1 < x \sum_d  \frac{\mu(d)^2}{d\varphi(d)} = cx
}
where $c > 1$, the last inequality by Lemma \ref{NonCMPhiSum}.  Thus,
\bs{ \sum_{n \le x} \frac {n^2} {\varphi(n)^2} &= \sum_{n \le x} \frac {n} {\varphi(n)} \sum_{d \mid n}  \frac{\mu(d)^2}{\varphi(d)} \le \sum_{d \le x} \frac{\mu(d)^2 d}{\varphi(d)^2} \sum_{n \le x/d} \frac n {\varphi(n)} \\
& < cx\sum_{d \ge 1} \frac{\mu(d)^2}{\varphi(d)^2} = c_1 x,
}
where the first inequality follows by using $\varphi(dn) \ge \varphi(d)\varphi(n)$ and the second by $\varphi(d) \gg d/\log \log d$. We conclude that for $z> y>1$,
\bs{\sum_{y < n \le z} \frac 1 {\varphi(n)^2} 
&= \int_y^z \frac 1 {x^2} d \sum_{n \le x} \frac {n^2} {\varphi(n)^2} = \frac 1 {z^2}\sum_{n \le z} \frac {n^2} {\varphi(n)^2} - \frac 1 {y^2} \sum_{n \le y} \frac {n^2} {\varphi(n)^2} \\
& \quad + 2\int_y^z x^{-3} \sum_{n \le x} \frac {n^2} {\varphi(n)^2} dx < \frac{2c_1-1}{y} + \frac 1 {y^2} - \frac{c_1} z.
}
Taking limit as $z \to \infty$, we get the result.
\end{proof}
\subsection{Proof of Theorem \ref{T:NonCMasymp}}
We shall assume $f < \tfrac 12\sqrt x$ since otherwise the theorem trivially holds. For a squarefree integer $d \ge 1$, put 
\eqs{\pi_{E,d}(x; f, a)=\#\{p\le  x:  p \nmid 2N_E, p \equiv a \mod f, p \text{ splits completely in } K_d \} .}
If a prime $p \le x$ splits completely in $K_d$ for some $d>1$, then $p$ splits completely in $K_q$ for each prime $q \mid d$. Since $p$ ramifies in $\cq p$ and $\cq p \subseteq K_p$ by \cite[Proposition 3.5\#3]{CoMu}, $p \nmid d$. Consequently, it follows from Lemmas \ref{cyclicity} and \ref{compofdivfields} that $d^2$ divides $|\E (\F_p)|$. Then, by Hasse's inequality $d^2 \le (\sqrt{p}+1)^2$, yielding $ d \le \sqrt{x}+1$. Hence, using inclusion-exclusion principle we can write 
\eqs{
\pE =\sum_{d \le \sqrt{x}+1}\mu(d)\pi_{E,d}(x; f, a). }
Put 
\eqn{\label{Sigma12}
\Sigma_1 = \sum_{d \le y} \mu(d)\pi_{E,d} (x; f, a), \qquad \Sigma_2 = \sum_{y < d \le  \sqrt{x}+1} \mu(d)\pi_{E,d}(x; f, a),
}
where $y$ is a parameter satisfying $2f \le y \le \sqrt x$.
\subsubsection{Main Term $\Sigma_1$}
For each square-free $d \le y$, there is a unique automorphism in $\gaq{K_d\cq f}$ whose restrictions to $K_d$ and $\cq f$ are identity and $\sigma_a$, respectively, provided that $\gamma_{a,f}(K_d) = 1$. Thus, $\pi_{E,d}(x; f, a)$ counts primes $p\le x$ of good reduction whose Frobenius automorphism coincides with this automorphism whenever $\gamma_{a,f}(K_d) = 1$. Therefore, it follows from Lemma \ref{effectiveChebo} that for each squarefree $d \le y$, 
\eqs{\pi_{E,d}(x;f,a) = \frac{\li(x)}{[K_d\cq f:\Q]} + O\biggl(x^{1/2} \log \Bigl( x[K_d\cq f:\Q] \prod_p p\Bigr) \biggr)}
if $\gamma_{a,f}(K_d)=1$, and is $0$ otherwise. Here, $p \in \cP(K_d\cq f)$. 

Note that $[K_d\cq f:\Q] \le [K_d:\Q]\varphi(f) < d^4 f$, the second inequality by \eqref{GalKm}. By \cite[Proposition 3.5\#3]{CoMu}, $\cq f \subseteq K_f$. Thus, $K_d \cq f \subset K_{[d,f]}$, and this implies $\cP(K_d\cq f/\Q) \subseteq \cP(K_{[d,f]}/\Q)$. By \cite[p.~179]{Sil86}, we conclude that $\cP(K_d\cq f/\Q)$ is a subset of primes dividing $dfN_E$. Therefore, the above error is $\ll x^{1/2} \log(dfxN_E)$, and we conclude
\eqn{\label{Sigma1part1}
\Sigma_1 = \li(x) \sum_{d \le y} \frac{\mu(d)\gamma_{a,f}(K_d)}{[K_d \cq f:\Q]}+ O(yx^{1/2}\log(fxN_E)).
}
Write $f = f_1 f_2$, where $f_1 \mid \me^\infty$ and $(f_2,\me)=1$. Then, 
\bs{
\sum_{d > y} \frac {\mu^2(d)} {[K_d\cq f:\Q]} 
&= \su{de > y\\ d \mid \me, (e,\me)=1} \frac {\mu^2(de)} {[K_{de}\cq f:\Q]} \\
&=\su{d \mid \me} \frac {\mu^2(d)} {[K_d \cq{f_1}:\Q]} \su{e > y/d \\ (e,\me)=1} \frac {\mu^2(e)} {[K_e \cq{f_2}:\Q]}\\
&\le \su{d \mid \me} \frac {\mu^2(d)} {\varphi(f_1)} \su{e > y/d \\ (e,\me)=1} \frac {\mu^2(e)[K_e \cap \cq{f_2}:\Q]} {[K_e:\Q][\cq {f_2}:\Q]}.
}
Here, the second equality follows by Lemma \ref{L:disjointness} (see the proof of Lemma \ref{densitysumprod} for details). By \cite[Prop. 3.6.2]{CoMu} and Lemma \ref{DivCycIntersection}, we get
\eqs{[K_e:\Q] \gg e^3\varphi(e), \qquad [K_e\cap \cq {f_2}:\Q] = \varphi(e,f_2).}
Thus, the last sum over $e$ is
\bs{
&\ll \frac 1 {\varphi(f_2)} \su{e > y/d} \frac {\mu^2(e)\varphi(e,f_2)} {\varphi(e)e^3} = \frac 1 {\varphi(f_2)} \sum_{k \mid f_2} \varphi(k) \su{e > y/d\\ (e,f_2)=k} \frac {\mu^2(e)} {\varphi(e)e^3} \\
&\le \frac 1 {\varphi(f_2)}\sum_{k \mid f_2} \frac 1 {k^3} \su{e > y/(kd)} \frac 1 {\varphi(e)e^3},
}
where, in the last inequality, we used $\varphi(ek) \ge \varphi(e) \varphi(k)$. By Lemma \ref{NonCMPhiSum} we derive that
\eqn{\label{Sigma1part2}
\sum_{d > y} \frac {\mu^2(d)} {[K_d\cq{f}:\Q]} 
\ll \frac{\tau(f_2)}{y^3\varphi(f)} \su{d \mid \me} \mu^2(d) d^3 \ll \frac{\tau(f_2)}{y^3\varphi(f)} \me^3.
}
\subsubsection{Estimate of the error $\Sigma_2$}
By Lemma \ref{cyclicity}, and the fact that $p$ splits completely in $\cq d$, we obtain 
\eqs{ 
\Sigma_2 \le \sum_{y<d \le \sqrt x +1} \su{p \le x, p\nmid 2N_E \\  p \equiv a \mod{f} \\  p \equiv 1 \mod{d} \\ d^2 \mid \#\E (\F_p) } 1.
} 
Writing $|\E(\F_p)| = p+1-a_p$, we have by Hasse's inequality, $|a_p| < 2 \sqrt p \le 2 \sqrt x$. Thus, $\Sigma_2$ is 
\bs{
&\le \sum_{y<d \le \sqrt x +1} \sum_{|b| \le 2\sqrt x} \su{p \le x, p\nmid 2N_E \\  p \equiv a \mod{f} \\  p \equiv 1 \mod{d} \\ d^2 \mid p+1-b \\ a_p = b } 1 \le \sum_{y<d \le \sqrt x +1} \su{|b| \le 2\sqrt x\\ d \mid b-2} \su{n \le x \\  n \equiv a \mod{f} \\ n \equiv b-1 \mod d^2} 1 \\
&\ll \sum_{y<d \le \sqrt x +1} \su{|b| \le 2\sqrt x\\ d \mid b-2 \\ (d^2,f)\mid a+1-b} \( 1 + \frac x {[f,d^2]} \) \ll \sum_{y<d \le \sqrt x +1} \(1 + \frac{\sqrt x} d \) \( 1 + \frac x {[f,d^2]} \) \\ 
&\ll \sqrt x \log x + \frac x f \sum_{y<d \le \sqrt x +1} \frac{(f,d^2)}{d^2} \(1 + \frac{\sqrt x} d \).
}
The last sum over $d$ is 
\bs{
&= \sum_{n \mid f} n \su{y<d \le \sqrt x +1\\ (f,d^2)=n} \frac 1 {d^2} \(1 + \frac {\sqrt x} d \) = \sum_{n \mid f} n \sum_{1 \le k \le n} \su{y<d \le \sqrt x +1\\ d \equiv k \mod n \\ (f,d^2)=n} \frac 1 {d^2} \(1 + \frac {\sqrt x} d \) \\
&\le \sum_{n \mid f} n \su{1 \le k \le n \\ n \mid k^2} \su{y < d \le \sqrt x +1 \\ d \equiv k \mod n} \frac 1 {d^2} \(1 + \frac {\sqrt x} d \).
}
Using the estimate
\eqs{ 
\su{d > y\\ d \equiv k \mod n} \frac{1}{d^\ell} 
%= \sum_ {m>(y-k)/n} \frac{1}{(mn+k)^\ell} 
< \frac 1 {n^\ell}\sum_{m>(y-k)/n} \frac 1 {m^\ell} \ll \frac 1 {n(y-n)^{\ell-1}} \qquad (\ell >1 ), 
}
and recalling that $2f \le y \le \sqrt x$, we obtain
\bs{\label{Sigma2}
\Sigma_2 &\le \sqrt x \log x + \frac x f\sum_{n \mid f} n \su{1 \le k \le n \\ n \mid k^2} \su{y < d \le \sqrt x +1 \\ d \equiv k \mod n} \frac 1 {d^2} \(1 + \frac {\sqrt x} d \) \\
& \ll \sqrt x \log x + \frac{x^{3/2}}{fy^2} H(f), 
}
where $H(f)$ is given by \eqref{E:H}. 
\subsubsection{Finale}
Combining \eqref{Sigma1part1}, \eqref{Sigma1part2} and \eqref{Sigma2}, we obtain
\eqs{\pE - \delta_E (a,f) \li(x) \ll  \frac{x\tau(f_2)\me^3}{y^3\varphi(f)\log x} + x^{1/2}y\log(fxN_E) + \frac{x^{3/2}}{fy^2} H(f).
}
By \cite[Lemma 2.4]{GraKol}, there is some $y$ 
%\bs{ 	y^{-3}\frac{x\nu(f_2)\me^3}{\varphi(f)\log x} 	+ y^{-4} \frac{x\tau(f_2) \me^4 (\log x)^{2/3}}{\varphi(f)\log x} + y^{-2} \frac{x^{3/2}H(f)}{f} + yx^{1/2}\log(fxN_E). }
in the interval $[2f,\sqrt x]$ for which the right hand side becomes
\bs{\ll &\frac{\tau(f_2)\me^3}{x^{1/2}\varphi(f)\log x} + x^{1/2} \frac{H(f)} f + x^{1/2}f\log(fxN_E)  \\
& + x^{5/8} \( \frac{\tau(f_2)\me^3 \log^3(fxN_E)}{\varphi(f)\log x} \)^{1/4}
+  x^{5/6} \( \frac{H(f) \log^2(fxN_E)} f \)^{1/3}.}
Note that writing $n = b^2 c$, where $b^2$ is the largest square dividing $n$, yields 
\eqn{\label{eq:Hbound}
\su{1 \le k \le n \\ n \mid k^2} 1 = \su{1 \le k \le b } 1 = b,}
and it follows that $H(f)$ is multiplicative. For $k \ge 1$, we have
\eqs{H(p^{2k}) = 2\sigma (p^{k-1}) + p^k, \quad H(p^{2k-1}) = 2\sigma (p^{k-1}).}
This gives the inequality in \eqref{E:bound4H}. In particular, $H(f) < f^2$ holds. Thus, the second term can be eliminated in the error term above, and we end up with \eqref{E:nonCMError}. This completes the proof.

\subsection{Positivity of Density $\delta_E(f,a)$}
Given a family 
\eqs{\cF =\{ L_p : \forall p, L_p \subseteq K_p, L_p/\Q \text{ is Galois} \},}
we define the density associated with $\cF$ by
\eqs{\delta_\cF (f,a) := \sum_{d \ge 1} \frac{\mu(d)\gamma_{a,f}(L_d)}{[L_d \cq f:\Q]},\qquad \text{with }L_d = \prod_{p \mid d} L_p, 
}
where, for any number field $L$, 
\eqs{%\label{gammaAandF}
\gamma_{a,f} (L) = \left\{ \begin{array}{l@{\quad}l}
1 & \text{if } \sigma_a \in \Gal(\cq f/L\cap \cq f),\\
0 & \text{otherwise.}
\end{array}\right.}
In particular, $\delta_E (f,a) = \delta_\cF (f,a)$ when $L_p = K_p$ for each $p$. 

\begin{lemma}\label{densitysumprod}
Let $\cF = \{ L_p \}_p$ be a family where $\Q \subsetneq L_p \subseteq K_p$ for each prime $p$. Then, $\delta_E (f,a) \ge \delta_\cF (f,a)$. Furthermore, if $L_p = K_p$ for each $p\nmid \me$, then  
\eqs{\delta_\cF (fg,a) = \frac 1 {\varphi(fg)} \prod_{\substack{p \nmid \me \\(p,f) \mid a-1}} \( 1 - \frac{\varphi(p,f)}{[K_p:\Q]}\) \sum_{d \mid \me} \frac{\mu(d)\gamma _{a,g}(L_d)}{[L_d:L_d\cap \cq g]},}
where $(f,\me)=1$, $g \mid \me^\infty$, and $(a,fg)=1$. 
\end{lemma}
\begin{rem}
For any prime $p \nmid A(E)$, $[K_p:\Q] = (p^2-p)(p^2-1)$, so the product is absolutely convergent.
\end{rem}
\begin{proof}
For any finite subset $\cP$ of primes, the set
\eqs{\{p\le x: p \nmid 2N_E, p \equiv a \mod f, \forall q\in \cP, p \text{ does not split completely in } K_q \}}
contains 
\eqs{\{p\le x: p \nmid 2N_E, p \equiv a \mod f, \forall q\in \cP, p \text{ does not split completely in } L_q \}.}
Thus, proceeding as in the proof of \cite[Lemma 6.1]{CoMu}, the first assertion follows.

As for the latter, we write		
\eqs{\delta_\cF (fg,a) = \sum_{d \mid \me} \su{e \\ 
(e,\me)=1} \frac{\mu(de)\gamma _{a,fg}(L_{de})}{[L_d K_e \cq {fg} :\Q ]}.
}
First note that 
\bs{[L_d K_e \cq {fg} :\Q ] 
&= \frac{[L_{de}:\Q][\cq{fg}:\Q]}{[L_{de}\cap \cq{fg}:\Q]}\\
&= \frac{[L_d:\Q][K_e:\Q][\cq{fg}:\Q]}{[L_d\cq{g}\cap K_e\cq f:\Q][K_e \cap \cq f:\Q][L_d\cap \cq g:\Q]},
}
and since numerators are the same, so are the denominators. Furthermore, since $(ef, dg\me)=1$, Lemma \ref{L:disjointness} gives
\eqs{L_d\cq{g}\cap K_e\cq f \subseteq K_{[d,g]} \cap K_{[e,f]} = \Q.} 
Thus, we have
\eqs{[L_{de}\cap \cq{fg}:\Q] = [K_e \cap \cq f:\Q][L_d\cap \cq g:\Q].}
Since $K_e \cap \cq f$ and $L_d\cap \cq g$ are disjoint by Lemma \ref{L:disjointness}, we see that 
\eqs{\gamma _{a,fg}(L_{de}) = 1 \Longleftrightarrow%iff $\sigma_a$ fixes the compositum $(K_e \cap \cq f)(L_d\cap \cq g)$ iff 
\gamma_{a,f}(K_e)= \gamma_{a,g}(L_d)=1.} 
Finally, since $K_e \cap \cq f = \cq{(e,f)}$ by Lemma \ref{DivCycIntersection}, $\delta_\cF(fg,a)$ is given by
\eqs{\frac 1 {\varphi(fg)} \sum_{d \mid \me} \frac{\mu(d)\gamma _{a,g}(L_d) [L_d\cap \cq g:\Q]}{[L_d:\Q]} \su{e \\ (e,\me)=1\\
(e,f) \mid a-1} \frac{\mu(e) \varphi(e,f)} {[K_e:\Q]},
}
and the result follows by writing the last sum as a product. 
\end{proof}
\begin{proof}[Proof of Theorem \ref{T:DensitynonCM}]
We choose $L_2 = K_2$, $L_p =\cq p$ for $p \mid \me/2$, $L_p =K_p$ for $(p,\me)=1$. By Lemma \ref{densitysumprod}, 
\eqn{\label{densitystep1}
\delta_E(f,a) \ge \delta_\cF (f,a) = \frac 1 {\varphi(f)} \prod_{\substack{p \nmid \me \\(p,f) \mid a-1}} \( 1 - \frac{\varphi(p,f)}{[K_p:\Q]}\)  \sum_{d \mid \me} \frac{\mu(d)}{[L_d:\Q ]}.	}
Splitting the sum over $d$, we obtain
\bs{\sum_{d \mid \me} \frac{\mu(d)}{[L_d:\Q ]}
&=\su{d \mid \me\\
2 \nmid d} \frac{\mu(d)}{[\cq d:\Q ]} - \su{d \mid \me/2 \\
2 \nmid d} \frac{\mu(d)}{[K_2 \cq d:\Q ]} \\
&= \su{d \mid \me\\
2 \nmid d} \frac{\mu(d)}{\varphi(d)} \( 1 - \frac {[K_2 \cap \cq d:\Q ]} {[K_2:\Q]} \) \\
&= \( 1 - \frac {[\ab 2 :\Q]} {[K_2:\Q]} \) \su{\ff_2 \mid d \mid \me\\ 2 \nmid d} \frac{\mu(d)}{\varphi(d)} + \( 1 - \frac 1 {[K_2:\Q]} \) \su{\ff_2 \nmid d \mid \me\\ 2 \nmid d} \frac{\mu(d)}{\varphi(d)}.  
}
Here, we have used the fact that $K_2 \cap \cq d = \ab 2 \cap \cq d$ is either $\Q$ or $\ab 2$. The latter implies $\ab 2 \subseteq \cq {(\ff_2,d)}$, which holds if $\ff_2 = (\ff_2, d)$; that is, if $\ff_2 \mid d$. The converse trivially holds. If $\ff_2$ is not square-free, then 
\eqs{\sum_{d \mid \me} \frac{\mu(d)}{[L_d:\Q ]} = \( 1 - \frac 1 {[K_2:\Q]} \) \prod_{2< p\mid \me} \( 1 - \frac{1}{p-1}\).}
If $\ff_2$ is square-free, then by \eqref{QuadConductor} and \eqref{CubicConductor}, it must be odd. Then, writing 
\eqs{\su{\ff_2 \nmid d \mid \me\\ 2 \nmid d} \frac{\mu(d)}{\varphi(d)}  = \su{d \mid \me\\ 2 \nmid d} \frac{\mu(d)}{\varphi(d)} - \su{\ff_2 d \mid \me\\ 2 \nmid d\\ (d,\ff_2)=1} \frac{\mu(d\ff_2)}{\varphi(d\ff_2)}
}
we derive
\bs{\sum_{d \mid \me} \frac{\mu(d)}{[L_d:\Q ]} &= \( 1 - \frac 1 {[K_2:\Q]} \) \su{d \mid \me\\ 2 \nmid d} \frac{\mu(d)}{\varphi(d)}  
- \frac {[\ab 2 :\Q]-1} {[K_2:\Q]} \su{\ff_2 d \mid \me\\ 2 \nmid d\\ (d,\ff_2)=1} \frac{\mu(d\ff_2)}{\varphi(d\ff_2)}. 
}
The second sum on the right side can be written as 
\bs{\su{\ff_2 d \mid \me\\ 2 \nmid d\\ (d,\ff_2)=1} \frac{\mu(d\ff_2)}{\varphi(d\ff_2)} 
&= \frac{\mu(\ff_2)}{\varphi(\ff_2)}\su{d \mid \me/\ff_2\\ 2 \nmid d} \frac{\mu(d)}{\varphi(d)} = \frac{\mu(\ff_2)}{\varphi(\ff_2)} \prod_{2 < p \mid \me/\ff_2} \( 1 - \frac{1}{p-1}\) \\
&= \mu(\ff_2) \frac{\prod_{2 < p \mid \me} \( 1 - \frac{1}{p-1}\)}{\varphi(\ff_2)\prod_{p \mid \ff_2} \( 1 - \frac{1}{p-1}\)} = \frac{\mu(\ff_2)}{\prod_{2 < p \mid \ff_2} ( p - 2 )} \su{d \mid \me\\ 2 \nmid d} \frac{\mu(d)}{\varphi(d)},}
where we have used the fact that $\me$ and $\ff_2$ are square-free (and, $\ff_2$ is odd). Inserting this expression back into the previous equation, we obtain
\bs{\sum_{d \mid \me} \frac{\mu(d)}{[L_d:\Q ]} 
&= \frac 1 {[K_2:\Q]} \( [K_2:\Q]- 1  - \frac{\mu(\ff_2)([\ab 2:\Q] -1)} {\prod_{2 < p \mid \ff_2} ( p - 2 )} \) \su{d \mid \me\\ 2 \nmid d} \frac{\mu(d)}{\varphi(d)}.
}
Combining this identity with \eqref{densitystep1}, we conclude that 
\mult{\delta_\cF (f,a) = \frac 1 {\varphi(f)} \prod_{\substack{p \nmid \me \\(p,f) \mid a-1}} \( 1 - \frac{\varphi(p,f)}{[K_p:\Q]}\)  \prod_{2< p\mid \me} \( 1 - \frac{1}{p-1}\) \\
\cdot \frac 1 {[K_2:\Q]} \( [K_2:\Q]- 1  - \frac{\mu(\ff_2)([\ab 2:\Q] -1)} {\prod_{2 < p \mid \ff_2} ( p - 2 )} \) >0, }
and this gives \eqref{E:densitynonCM}.
\end{proof}

\section{Proofs of Theorems \ref{T:CMasymp} and \ref{T:DensityCM}}
Throughout this section, we assume that $E$ is an elliptic curve over $\Q$ with complex multiplication.
\subsection{Proof of Theorem \ref{T:CMasymp}}
We proceed as in the proof of Theorem \ref{T:NonCMasymp}. Everything up to equation \eqref{Sigma1part1} applies to the CM case. We start with the estimate of $\Sigma_1$ given by \eqref{Sigma12}. By \cite[Proposition 3.8]{CoMu}, $[K_d:\Q] \gg \varphi(d)^2$. Thus, using Lemma \ref{CMPhiSum} we obtain
\eqs{
\sum_{d > y} \frac {\mu^2(d)} {[K_d\cq f:\Q]} 
\ll \sum_{d > y} \frac 1 {\varphi(d)^2} \ll y^{-1},}
which yields
\eqn{\label{CMSigma1}
\Sigma_1 = \li(x) \delta_E(f,a) + O\Bigl(\frac{x}{y\log x} + yx^{1/2}\log(fxN_E)\Bigr).
}

Next, we deal with 
\eqs{\Sigma_2 = \sum_{y < d \le  \sqrt{x}+1} \mu(d)\pi_{E,d}(x; f, a).}
If $p$ is a prime counted in $\pi_{E,d}(x; f, a)$, then $p$ splits completely in $K_d$ and thus in $\cq d$ since $\cq d  \subseteq K_d$. Thus, by Lemma \ref{cyclicity}, $d^2$ divides $|\E(\F_p)|$ and also $d \mid p-1$. Hence, we note that $|\E(\F_p)| \neq p+1$, since otherwise, $d \mid p+1-(p-1)=2$, which is impossible since $d > y>2$. This means no prime except possibly $p=3$ that splits completely in $K_d$ can have supersingular reduction. Therefore, it follows from \cite[Lemma 2.2]{Co2} that $p\neq 3$ splits completely in $K_d$ if and only if $\pi_p - 1 \in d\fO_K$. %Furthermore, by \cite[Lemma 2.3]{Co2}, $K= \Q(\pi_p)$ for every prime $p$ of ordinary good reduction. 
Here, $\pi_p$ is one of the complex roots of the polynomial $X^2 - (p+1-|\E(\F_p)|)X+p$. Note that $N_{K/\Q}(\pi_p) = \pi_p \overline{\pi_p} = p$. Thus, we deduce that
\eqs{\pi_{E,d}(x;f,a) \le 1 + \big|\{ 3 \neq p\le x : p \nmid N_E, p \equiv a \mod f, \pi_p \equiv 1 \mod d\fO_K \}\big|.}

Since $K$ is an imaginary quadratic extension of $\Q$, $K=\Q(\sqrt{-D})$ for some square-free positive integer $D$, and $\fO_K = \Z[\omega_D]$, where
\eqs{\omega_D = \left\{ 
\begin{array}{r@{\quad\text{if }}l}
\sqrt{-D} & D \equiv 1,2 \mod 4\\
\frac 1 2 (1+ \sqrt{-D}) & D \equiv 3 \mod 4. 
\end{array} \right.
}
Thus, any $\alpha \in \fO_K$ with $\alpha \equiv 1 \mod d\fO_K$ can be written as
%\eqs{N_{K/\Q}(\alpha) = (1+bd)^2 + (1+bd)cd +\frac{1+D}{4} (cd)^2}  Then,  \eqs{4 N_{K/\Q}(\alpha) = (2+(2b+c)d)^2 + D(cd)^2} Solving $N(\alpha) \equiv a \mod p^k$ for an odd prime $p$  is equivalent to solving $4N(\alpha) \equiv 4a \mod p^k$. In this case, it is enough to solve  \eqs{(2+ed)^2 + D(cd)^2 \equiv 4a \mod p^k} since $2b+c \equiv e \mod p^k$ has a unique solution $b$ for each $c$ and $e$. 
\eqs{\alpha = \left\{ 
\begin{array}{c@{\quad\text{if }}l}
bd+1 + cd\sqrt{-D} & D \equiv 1,2 \mod 4\\
\frac 1 2 \(bd+ 2 + cd \sqrt{-D}\),\, b \equiv c \mod 2 & D \equiv 3 \mod 4, 
\end{array} \right.
}
for some integers $b$ and $c$, and therefore has its norm equal to 
\eqs{N_{K/\Q}(\alpha) = \left\{ 
\begin{array}{r@{\quad\text{if }}l}
(bd+1)^2 + D(cd)^2 & D \equiv 1,2 \mod 4\\
\tfrac 1 4 \((bd+2)^2 + D(cd)^2\)  & D \equiv 3 \mod 4. 
\end{array} \right.
}

Note that 
\eqs{N_{K/\Q}(\pi_p) \equiv a \mod f \Longleftrightarrow 4N_{K/\Q}(\pi_p) \equiv 4a \mod ((f,2)^2f).}
We shall use this equivalent form only when $D \equiv 3 \mod 4$ since, in this case, $4N_{K/\Q}(\alpha)$ becomes a quadratic form in $b,c,d$ with \emph{integer} coefficients. Using this observation we deduce that $\pi_{E,d} (x;f,a)$ is at most
\eqs{|\{ (b,c) \in \Z^2 : F(b,d,c) \equiv a' \mod f', F(b,d,c) \le 4x, 2\mid b-c \text{ if } D \equiv 3 \mod 4\}|,} 
where
\eqs{\left\{ 
\begin{array}{l@{\quad\text{if }}l}
F(b,d,c) = (bd+1)^2 + D(cd)^2, a'=a, f'=f & D \equiv 1, 2 \mod 4 \\
F(b,d,c)=(bd+2)^2 + D(cd)^2, a'=4a, f'=(f,2)^2f & D \equiv 3 \mod 4.
\end{array}	
\right. }

Now, summing over $d \in (y,\sqrt x +1]$ leads to the bound
\bs{\Sigma_2
&\le \su{\alpha, \beta, \gamma \mod f'}\, 
\su{y< d\le \sqrt x+1\\ d\equiv \beta \mod f'}\,
\su{b\equiv \alpha, c \equiv \gamma \mod f'\\ F(b,d,c) \le 4x \\ F(b,d,c) \equiv a' \mod f'\\ (b \equiv c\mod 2)} 1 \\
&\le \su{\alpha, \beta, \gamma \mod f'\\F(\alpha,\beta,\gamma)\equiv a' \mod f'\\ (\alpha \equiv \gamma \mod 2)}\, \su{y< d\le \sqrt x+1\\ d\equiv \beta \mod f'} \,
\su{|b| \le \frac{2\sqrt x+2}{d}\\b\equiv \alpha \mod f'}\, \su{|c| \le \frac{2\sqrt x}{d\sqrt D}\\c \equiv \gamma \mod f'} 1,
}
with the parity condition required only when $D \equiv 3 \mod 4$. Note that the second inequality follows from the fact that
\eqs{F(b,d,c) \equiv F(b \mod f',d \mod f',c \mod f') \mod f'}
since $F(b,d,c)$ has integer coefficients.

For $y \in [2f,\sqrt x]$, and \emph{uniformly} for any $\alpha, \beta, \gamma$ modulo $f$,
\bs{\su{y< d\le \sqrt x+1\\ d\equiv \beta \mod f'} \,
\su{|b| \le \frac{2\sqrt x+2}{d}\\b\equiv \alpha \mod f'}\, \su{|c| \le \frac{2\sqrt x}{d\sqrt D}\\c \equiv \gamma \mod f'} 1
&\ll \su{y< d\le \sqrt x+1\\ d\equiv \beta \mod f'} \(1 + \frac{\sqrt x}{df}\) \( 1 + \frac{\sqrt x}{df\sqrt D}  \) \\
&\ll \su{y< d\le \sqrt x+1\\ d\equiv \beta \mod f'} \(1 + \frac{\sqrt x}{df} + \frac{\sqrt x}{df\sqrt D} + \frac{x}{d^2f^2\sqrt D} \) \\
&\ll_D \frac{\sqrt x}{f} + \frac{\sqrt x \log x}{f^2}+ \frac{x}{yf^3}. 
}
Note that the implied constant depends on $K$. Since $E/\Q$ has CM by $\fO_K$, then $K$ is one of the nine imaginary quadratic fields of class number one, and so the implied constant above can be replaced by an absolute constant. Inserting this estimate into the previous estimate of $\Sigma_2$, we deduce that
\eqn{\label{CMSigma2}
\Sigma_2 \ll \(\frac{\sqrt x}{f} + \frac{\sqrt x \log x}{f^2}+ \frac{x}{yf^3}\) G_D(a,f),}
where $G_D(a,f)$ is the cardinality of the set 
\eqn{\label{Gaf}
\{ (\alpha,\beta, \gamma) \in (\Z/f'\Z)^3 : F(\alpha,\beta,\gamma) \equiv a' \mod f', 2 \mid \alpha - \gamma \text{ if } D\equiv 3 \mod 4 \}.}
Combining \eqref{CMSigma1} and \eqref{CMSigma2} we obtain the bound
\bs{\pE - \delta_E (a,f) \li(x) 
&\ll  x^{1/2}y\log(fxN_E) + \frac{x}{y\log x} + \frac{x}{yf^3} G_D(a,f)\\ 
&\quad+ x^{1/2}\(\frac 1 f + \frac{\log x}{f^2} \)G_D(a,f).}
Recalling that $2f \le y \le \sqrt x$ and using \cite[Lemma 2.4]{GraKol} yields the error 
\bs{
E(x) 
&\ll x^{1/2}f\log(fxN_E) + x^{1/2}\frac{G_D(a,f)}{f^3}  + x^{3/4} \( \frac{\log(fxN_E)}{\log x}\)^{1/2} \\
&\quad + x^{3/4} \( \frac{\log(fxN_E) G_D(a,f)}{f^3} \)^{1/2} +  x^{1/2}\(\frac 1 f + \frac{\log x}{f^2} \)G_D(a,f).}
Note that the second term can be eliminated since it is already smaller than the fifth term, and this gives the error in \eqref{CMError}. 

To complete the proof of Theorem \ref{T:CMasymp}, we need to estimate $G_D(a,f)$. Since $G_D$ is multiplicative in the second variable, it is enough to estimate $G_D(a,p^k)$ for primes $p$ with $p^k \| f'$. Note that $p\nmid a$ since $(a,f)=1$. 

Assume first that $D \equiv 1,2 \mod 4$. Recall, in this case, $f'=f$ and $a'=a$. Put
\eqs{A_i = \{ (\alpha, \beta, \gamma) : p^i \| a - D(\beta \gamma)^2, F(\alpha,\beta,\gamma) \equiv a \mod p^k \}.}

Note that for any triple in $A_i$ with $i\ge 1$, $p \nmid D\beta\gamma$. Also, if $i \ge k$, then for $\varphi(p^k)$ possible choices of $1\le \gamma \le p^k$, there are at most $\eta(p^k)$ choices for $\beta$ satisfying 
\eqs{D(\beta\gamma)^2 \equiv a \mod p^k,}
where $\eta(p^n)= 2$ if $p$ is odd, or $p=2$ and $n=1,2$, and it equals $4$ otherwise. Furthermore, 
\eqs{
(\alpha\beta+1)^2 \equiv a - D(\beta \gamma)^2 \equiv 0 \mod p^k
}
implies 
\eqs{\alpha\beta \equiv -1 \mod p^{\ceil{k/2}},}
and there is unique $\alpha$ modulo $p^{\ceil{k/2}}$ satisfying this congruence, which gives $p^{k-\ceil{k/2}}$ choices for $\alpha$ modulo $p^k$. Hence, 
\eqn{\label{Step1}
\sum_{i \ge k} |A_i| \le \eta(p^k)p^{k-\ceil{k/2}}\varphi(p^k).}

Next, assume that $p\nmid a- D(\beta\gamma)^2$. Then,
\eqs{X^2 \equiv a - D(\beta \gamma)^2 \mod p^k}
has at most $\eta(p^k)$ solutions. If $X_0 = X_0(\beta,\gamma)$ is one of these solutions, and $p^i \| \beta$ with $0 \le i \le k$, then there are $\gcd(\beta,p^k)=p^i$ values of $\alpha \in [1,p^k]$ satisfying
\eqs{\alpha \beta \equiv X_0-1 \mod p^k,}
provided $p^i \mid X_0-1$. Since there are $\varphi(p^{k-i})$ values of $\beta$ modulo $p^k$ with $p^i \| \beta$, and at most $p^k$ values of $\gamma$, we get
\eqn{\label{Step2}
|A_0| \le \eta(p^k) p^{2k} + \sum_{0 \le i \le k-1} \eta(p^k) p^k \varphi(p^{k-i}) p^i = \eta(p^k) p^{2k} \(k(1-1/p) + 1 \).
}

Finally, assume $1 \le i \le k-1$ and $k>2$ (note for $k \le 2$, this part will not contribute as will be seen below). In this case, we have
\eqs{D(\beta\gamma)^2 \equiv a \mod p^i.}
For $\varphi(p^k)$ choices of $\gamma$, there are at most $\eta(p^i) p^{k-i}$ choices for $\beta$ modulo $p^k$. For these values of $\gamma$ and $\beta$, 
\eqn{\label{Sqmodptok}
X^2 \equiv a - D(\beta \gamma)^2 \mod p^k
}
implies $p^{\ceil{i/2}} \mid X$, which then yields $p^{i+1} \mid a-D(\beta\gamma)^2$ if $i$ is odd. Thus, \eqref{Sqmodptok} has no solutions for odd $i<k$. Otherwise, writing $X = p^{i/2} Y$ with $1 \le Y \le p^{k-i/2}$ gives
\eqs{Y^2 \equiv \frac{a-D(\beta\gamma)^2}{p^i} \mod p^{k-i}.}
Since the right side is now coprime to $p$, there are at most
$\eta(p^{k-i})$ solutions for $Y$ modulo $p^{k-i}$, which gives $\eta(p^{k-i})p^{i/2}$ choices for $X$. If $X_0$ is one of these possible solutions, then
\eqs{\alpha \beta +1 \equiv X_0 \mod p^k}
has exactly one solution for $\alpha$. Hence,
\bsc{\label{Step3}
\sum_{1 \le i \le k-1} |A_i| 
&\le \su{1 \le i \le k-1\\ 2 \mid i} \varphi(p^k) \eta(p^i) \eta(p^{k-i}) p^{k-i} p^{i/2}\\
&< \eta(p^k)^2 \varphi(p^k) \sum_{1 \le i \le \fl{(k-1)/2}} p^{k-i} < \eta(p^k)^2 p^{2k-1}.} 

Combining \eqref{Step1}, \eqref{Step2} and \eqref{Step3}, we conclude that
\bsc{\label{D12}
G_D(a,p^k) &\le \eta(p^k) p^{2k} \Bigl( \min\{1,(k-2)(k-1) \}\eta(p^k) p^{-1} + p^{-\ceil{k/2}} ( 1 - 1/p) \\
&\quad + k ( 1- 1/p) + 1 \Bigr) < 2k\eta(p^k) p^{2k}. 
}

Next, assume $D \equiv 3 \mod 4$. We shall count the solutions to 
\eqs{F(\alpha,\beta,\gamma) = (\alpha \beta + 2)^2 + D(\beta\gamma)^2 \equiv 4a \mod p^k.}
Assume first that $p$ is odd. Since $p \nmid 4a$ in this case, the proof in the previous case goes through and gives the same upper bound in \eqref{D12} for $G_D(a,p^k)$. 

Next, assume $2^k \| f$. Then, we consider $F \equiv 4a \mod 2^{k+2}$ with $\alpha \equiv \gamma \mod 2$. If $\gamma$ is even, then so is $\alpha$ and we have to count the solutions to 
\eqs{
(\alpha \beta + 1)^2 + D(\beta\gamma)^2 \equiv a \mod 2^k,}
where $\alpha, \gamma \in [1,2^{k+1}]$ and $\beta \in [1,2^{k+2}]$. When all variables lie in $[1,2^k]$, there are at most $2k \eta(2^k)2^{2k}$ triples by \eqref{D12}. Lifting variables, we get at most $32k \eta(2^k)2^{2k}$ solutions. 

When $\alpha$ and $\gamma$ are odd and $\beta$ is even, we end up with the congruence
\eqs{(\alpha \beta + 1)^2 +D(\beta\gamma)^2 \equiv a  \mod 2^k,}
where $\alpha, \gamma \in [1,2^{k+2}]$ are odd, while $\beta \in [1,2^{k+1}]$. If $\beta$ is odd, 
\eqs{\gamma^2 \equiv D^{-1} \beta^{-2} \(a - (\alpha \beta + 1)^2 \)\mod 2^k}
has at most $\eta(2^k)$ solutions for $\gamma$ since right hand is odd, and these can be lifted to $4\eta(2^k)$ solutions mod $2^{k+2}$. Hence, there are at most $4\eta(2^k) 2^{2k+1}$ triples modulo $2^{k+2}$. 

If $2^i \| \beta$ for $1 \le i \le k$, then 
\eqs{X^2 \equiv a - D(\beta\gamma)^2 \mod 2^k}
has at most $\eta(2^k)$ solutions. If $X_0$ is one of the possible solutions, then 
\eqs{\alpha \beta \equiv X_0 -1 \mod 2^k}
has at most $2^{i+2}$ solutions for $\alpha$ modulo $2^{k+2}$. There are $2^{k+1-i}$ values of $\beta$ modulo $2^{k+2}$ with $2^i \mid \beta$, and $2^{k+1}$ odd values of $\gamma \in [1,2^{k+2}]$. Hence, we get at most 
\eqs{4\eta(2^k) 2^{2k+1} + \sum_{1 \le i \le k} \eta(2^k) 2^{i+2+k+1-i+k+1} = (8+16k)\eta(2^k) 2^{2k}}
solutions.

Finally, if all the variables are odd, then we have
\eqs{\gamma^2 \equiv D^{-1}\beta^{-2} \(4a - (\alpha \beta + 2)^2\)  \mod 2^{k+2}.}
Given odd $\alpha, \beta \in [1,2^{k+2}]$, there are at most $\eta(2^{k+2})$ solutions for $\gamma \in [1,2^{k+2}]$ since the right hand side is odd. Hence, we obtain at most $\eta(2^{k+2}) 2^{2k+2}$ triples. Combining all the estimates, we deduce that
\eqs{G_D(a,2^k) \le \eta(2^k)2^{2k} (48k+16) < \frac{49}{2} \cdot 2k \eta(2^k) 2^{2k}.}	
Multiplying the bounds for $G_D(a,p^k)$ over the prime powers dividing $f$, we obtain the bound in \eqref{Gbound4f}. This completes the proof.
\subsection{Proof of Theorem \ref{T:DensityCM}}
Recall that End$_{\overline{\Q}}(E) \simeq \fO_K$, where $K=\Q(\sqrt{-D})$. By \cite[Lemma 6]{Mu1979}, for all $p \ge 3$, $K \subset K_p$.
Suppose first that $K_2 \cap K = \ab2 \cap K = \Q$ and that
\eqn{\label{spongebob}
\gamma_{a,f}(K_2K) = \gamma_{a,f}(K_2)\gamma_{a,f}(K).} 
Note that 
\bs{[K_2\cap \cq f:\Q][K\cap \cq f:\Q] &= [(K_2\cap \cq f)(K\cap \cq f):\Q] \\
&\le [K_2K \cap \cq f:\Q]}
since 
\eqs{(K_2 \cap \cq f)(K \cap \cq f) \subseteq K_2K \cap \cq f.}
Then, taking $\cF = \{K_2, K\}$ and using \cite[Lemma 6.1]{CoMu} yields 
\bs{\delta_\cF (a,f) 
&= \frac 1 {\varphi(f)} - \frac{\gamma_{a,f}(K_2)}{[K_2\cq f:\Q]} -\frac{\gamma_{a,f}(K)}{[K\cq f:\Q]} + \frac{\gamma_{a,f}(K_2)\gamma_{a,f}(K)}{[K_2K\cq f:\Q]} \\
&\ge \frac 1 {\varphi(f)} \Bigl( 1 - \frac{\gamma_{a,f}(K_2)[K_2\cap \cq f:\Q]}{[K_2:\Q]} \Bigr) \Bigl( 1 - \frac{\gamma_{a,f}(K)[K\cap \cq f:\Q]} 2\Bigr).
}
Thus, $\delta_\cF > 0$ if $K_2 \subsetneq \cq f$ or $\gamma_{a,f}(K_2) =0$, and $K \subsetneq \cq f$ or $\gamma_{a,f}(K)=0$, provided \eqref{spongebob} holds and $K_2 \cap K =\Q$.

If $\ab 2 = K$, then taking $\cF = \{ K_2 \}$ yields
\eqs{\delta_\cF (a,f) = \frac 1 {\varphi(f)} \Bigl( 1 - \frac{\gamma_{a,f}(K_2)[K_2\cap \cq f:\Q]}{[K_2:\Q]} \Bigr).
}
We conclude again that $\delta_\cF > 0$ if $K_2 \subsetneq \cq f$ or $\gamma_{a,f}(K_2) =0$.

%*************************************************
\appendix
\section{Intersections of Division Fields}
\begin{center}
\bf By Ernst Kani
\end{center}

\setcounter{theorem}{0}
\setcounter{prop}{0}
\setcounter{lemma}{0}
\setcounter{corollary}{0}

\noindent Let $E/K$ be an elliptic curve defined over a number field $K$. 
Recall that for each integer $m \ge 1$ we have a natural representation
\eqs{\rho_m = \rho_{E/K,m} : G_K=\Gal(\overline K/K) \longrightarrow \gl m := \GL_2(\Z/m\Z).}
The fixed field of its kernel is the $m$-division field $K(E[m]) = \overline K^{\,\ker(\rho_{m})}$, so
\eqs{\Gal (K(E[m])/K) \simeq G_m := \Im(\rho_{m}).}
Put
\eqs{S_{E/K} = \{ p \text{ prime}: G_p \neq \gl p \}.}
By Serre \cite{Serre72}, $S_{E/K}$ is finite if (and only if) $E$ is non-CM, which we assume 
henceforth. In this case the \emph{Serre constant} of $E/K$ is defined as the number
%\begin{equation}
%\label{A(E)}
\eqs{A_{E/K} = 30 \pr{p>5\\p \in S_{E/K}} p.}

The main aim of this appendix is to prove the following result.

\begin{theorem} \label{A:mainthm}
Let $E/\Q$ be a non-CM elliptic curve, and let $m, n\ge 1$ be integers with 
$(m, nN_EA_{E/\Q}) = 1$, where $N_E$ denotes the conductor of $E/\Q$. Then,
\bs{\Q(E[m]) \cap \Q(E[n]) = \Q.}
\end{theorem}

Note that we cannot drop the condition of Theorem \ref{A:mainthm} that $(m,N_E)=1$, even if $m$ is a prime; cf.~Proposition \ref{A:prop2} and Example \ref{A:ex} below. 

As we shall see presently, Theorem \ref{A:mainthm} follows from the following result which is valid for elliptic curves over an arbitrary number field $K$. This, in turn, follows easily from the results of the Appendix of \cite{Co3}.    

\begin{theorem} \label{A:thm2}
Let $E/K$ be a non-CM elliptic curve, and let $m, n\ge 1$ be integers with $(m, nA_{E/K}) = 1$. Then, $K(E[m])\,\cap\,K(E[n])$ is an abelian extension of $K$.
\end{theorem}   

\begin{proof}[Proof of Theorem \ref{A:mainthm} (using Theorem \ref{A:thm2})]
Put $L = \Q(E[n]) \cap \Q(E[m])$. By Theorem \ref{A:thm2} we know that $L/\Q$ is an abelian extension with $L \subset \Q(E[m])$. Since $m$ is coprime to $A_{E/\Q}$, we know that $\cq m$ is the maximal abelian extension of $\Q$ in $\Q(E[m])$; cf.~Corollary 13 of the Appendix of \cite{Co3}. Thus, $L \subset \cq m$, and so $L/\Q$ is ramified only at the primes $p \mid m$. On the other hand, since $L \subset \Q(E[n])$, we see by the criterion of N\'eron-Ogg-Shafarevi\v c that $L/\Q$ is ramified only at primes $p \mid nN_E$; cf.~Silverman \cite[Theorem VII.7.1]{Sil86}. Thus, since $(m,nN_E) =1$, it follows that $L/\Q$ is everywhere unramified and so $L=\Q$, as claimed.
\end{proof}

To prove Theorem \ref{A:thm2}, we will use some basic facts about the non-abelian composition factors of a subgroup $G$ of $\gl m$ which were presented in the Appendix of \cite{Co3}. For this, let $\cN(G)$ denote the set of (isomorphism classes) of non-abelian composition factors of a group $G$, and put
\eqs{\Occ(G) = \bigcup_{H\le G} \cN(H). }
\begin{prop} \label{A:prop1} \emph{(a)} For any integer $m>1$, we have that
	\eqs{\Occ(\GL_2(\Z/m\Z)) = \Occ(\SL_2(\Z/m\Z)) = \bigcup_{p|m}\Occ(\PSL_2(p)), }
	where $\PSL_2(p) = \SL_2(\Z/p\Z)/\{\pm1\}$, if $p$ is prime. Moreover, 
	$\Occ(\PSL_2(p)) = \emptyset$ when $p=2$ or $3$, whereas for $p \ge 5$ we have
	\eqs{\{\PSL_2(p)\} \subseteq \Occ(\PSL_2(p)) \subseteq \{A_5, \PSL_2(p)\}.}
	\emph{(b)} If $G \le \GL(m)$, where $(m,30)=1$, then 
	\eqs{G \ge \SL(m) := \SL_2(\Z/m\Z)\,\Leftrightarrow\, \forall p\mid m, \PSL_2(p) \in \Occ(G).  }
	If this is the case, then $G/\SL(m)$ is abelian and $\cN(G) = \{\PSL_2(p): p|m\}$. 
\end{prop}
\begin{proof}
(a) This is Lemma 10 of the Appendix of \cite{Co3}.

(b) The first assertion is Theorem 2(b) of the same Appendix. To prove the others, note that
$G/\SL(m) \le \GL(m)/\SL(m) \simeq (\Z/m\Z)^\times$ is abelian, so
\eqs{\cN(G) = \cN(\SL(m)) = \bigcup_{p \mid m} \cN \bigl( \SL(p^{v_p(m)}) \bigr),} 
the latter because $\SL(m) = \prod_{p|m} \SL(p^{v_p(m)})$. 
Since the kernel of the homomorphism $\SL(p^r) \rightarrow \SL(p)$ is a $p$-group, we have that 
\eqs{\cN(\SL(p^r)) = \cN(\SL(p)) = \{\PSL_2(p)\},} 
and so the last assertion follows. 
\end{proof}

\begin{corollary} \label{A:cor}
If $(m,A_{E/K})=1$, then $\sl m \le G_m$. Thus, if $L/K$ is a solvable extension with $L \subset K(E[m])$, then $L/K$ is abelian. 
\end{corollary}
\begin{proof} 
Since $(m, A_{E/K}) = 1$, we have that $G_p = \gl p$ for all $p\mid m$, and so $\PSL_2(p)\in \Occ(\GL(p))\subset \Occ(G_m)$, the latter because $G_p$ is a quotient of $G_m$, $\forall p \mid m$. Thus, $\sl m \le G_m$ by Proposition \ref{A:prop1} because $(m,30)=1$.  

To prove the second assertion, let 
\eqs{H:=\Gal(K(E[m])/L) \norm G:=\Gal(K(E[m])/K).} 
Since $G/H \simeq \Gal(L/K)$ is solvable and $G \simeq G_m$, we have that $\Occ(H) = \Occ(G_m)$. Thus, by Proposition \ref{A:prop1}(b) there exists $H_1\le H$ with $H_1 \simeq \sl m$, and then $G/H_1$ is abelian. Thus, the quotient $G/H$ of $G/H_1$ is also abelian.      
\end{proof}

\begin{proof}[Proof of Theorem \ref{A:thm2}]
Put $L = K(E[n]) \cap K(E[m])$ and $H = \Gal(L/K)$. Then $H$ is a quotient of 
$\Gal(K([E[n])/K) \simeq G_n \le \GL(n)$ and also of $\Gal(K(E[m])/K) \simeq G_m$, so
\bs{\cN(H) &\subset \Occ(\GL(n)) \cap \cN (G_m)\\
	&\subset \left( \{A_5\} \cup \{ \psl p : p\mid n, p \ge 5\} \right) \cap \{ \psl p: p \mid m\},}
where the last inclusion follows from both parts of Proposition \ref{A:prop1} together with 
Corollary \ref{A:cor}. Since $(n,m)=1$ and $5\nmid m$, we see that this intersection is empty 
because $\PSL(p) \simeq A_5 \Leftrightarrow p=5$ and $\PSL(p) \simeq \PSL(q) \Leftrightarrow p =q$;
cf.\ Lemma 3 of the Appendix of \cite{Co3}. Thus, $\cN(H)= \emptyset$, which means that 
$H$ is solvable. Since $L \subset K(E[m])$, we have by Corollary \ref{A:cor} that $L/K$ is abelian.
\end{proof}

We now show that the condition $(m,N_E)=1$ in Theorem \ref{A:mainthm} cannot be dropped. This follows from 
the following result together with Example \ref{A:ex} below which shows that there exist elliptic curves 
$E/\Q$ satisfying the hypotheses of Proposition \ref{A:prop2}.

\begin{prop}
\label{A:prop2}
Let $E/\Q$ be an elliptic curve with prime conductor $N_E = p$ with $p \equiv 3 \mod 4$. Suppose that 
the discriminant of some integral model of $E/\Q$ satisfies $\Delta_E <0$ and $v_p(\Delta_E) \equiv 1 \mod 2$. Then, $(p,A_{E/\Q})=1$, but 
\eqs{\Q(E[p]) \cap \Q(E[2]) = \Q(\sqrt{-p}). }
\end{prop}

\begin{proof} Since there are no elliptic curves of conductor $N_E < 11$, the hypothesis
implies that $p \ge 11$. Moreover, since $N_E$ is squarefree, $E/\Q$ is semi-stable (and non-CM), 
so by Corollary 1 of \S5.4 of Serre \cite{Serre72}, we know that $p\notin S_{E/\Q}$ because
$p > (\sqrt2 +1)^2 \approx 5.8$. Thus $p \nmid A_{E/\Q}$.  

For any integral model of $E/\Q$, there exists an integer $d\ge 1$ such that 
\eqs{\Delta_E = d^{12}\Delta_{E/\Q}^\textsuperscript{min},}
where $\Delta_{E/\Q}^\textsuperscript{min}$ denotes the minimal discriminant of $E/\Q$. Thus, the given conditions on $\Delta_E$ do not depend on the choice of the model. 

Since $N_E$ and  $\Delta^\textsuperscript{min}_{E/\Q}$ %of $E/\Q$
have the same prime divisors, we see that $\Delta^\textsuperscript{min}_{E/\Q} = -p^k$, with $k$ odd, so $\Delta_E = -d^{12}p^k$. 
By taking an integral model of the form $Y^2 = f(X)$, where $f(X)$ is a cubic,
we see that $\Q(E[2])$ is the splitting field of $f(X)$. Since $\Delta_E = 16\,\text{disc}(f)$,
it follows from field theory that $\Q(\sqrt{-p}) \subset \Q(E[2])$. Moreover, 
$\Q(\sqrt{-p})$ is the maximal abelian extension of $\Q$ in $\Q(E[2])$. Indeed, 
if $f(X)$ is irreducible, then this is clear by field theory, and otherwise we have that
$\Q(E[2]) = \Q(\sqrt{-p})$ is abelian.  

On the other hand, the condition $p\equiv 3 \mod 4$ implies (cf.~\cite[Theorem V1.3.3]{Lang}) that 
\eqs{\Q(\sqrt{-p}) \subset \cq p \subset \Q(E[p]).}
This proves the inclusion $\Q(\sqrt{-p}) \subset \Q(E[p]) \cap \Q(E[2])$. Since the latter intersection is abelian by Theorem \ref{A:thm2} and is contained in $\Q(E[2])$, it follows from what was said above that it is contained in $\Q(\sqrt{-p})$, and so the assertion follows. 
\end{proof}

\begin{ex} \label{A:ex} \em
Consider the following elliptic curves $E_i/\Q$ defined by the equations
\bs{E_1 : Y^2 &= X^3 - 432X + 8208,\\ 
	E_2: Y^2 &= X^3 - 432X + 15120\\
	E_3: Y^2 &= X^3 - 997056X - 383201712. }

The discriminant of $E_i$ is $\Delta_{E_i} = -6^{12}p_i$, for $i=1,2,3$, 
where $p_1=11$, $p_2=43$ and $p_3=19$. Furthermore, $N_{E_i} = p_i \equiv 3 \mod 4$, and so $E_i/\Q$ satisfies the hypotheses of Proposition \ref{A:prop2} with $p=p_i$, for $i=1,2,3$.
\end{ex}

%************************************

%\nocite{*}

\bibliographystyle{amsplain}

\end{document}